\documentclass[12pt]{amsart}
\usepackage{amssymb}
\usepackage{hyperref}
\usepackage[margin=1.25in]{geometry}
\usepackage[utf8]{inputenc}
\usepackage{chngcntr}
\usepackage{apptools}
\usepackage{thm-restate}
\usepackage{url}
\AtAppendix{\counterwithin{theorem}{section}}
\usepackage{filecontents}

\newtheorem{theorem}{Theorem}[]
\newtheorem{corollary}{Corollary}[]
\newtheorem{proposition}{Proposition}[section]
\newtheorem{lemma}[proposition]{Lemma}
\newtheorem{corollaryy}[proposition]{Corollary}
\newtheorem{claim}[proposition]{Claim}
\theoremstyle{definition}
\newtheorem{definitionn}[proposition]{Definition}
\theoremstyle{remark}
\newtheorem{remark}[proposition]{Remark}

\theoremstyle{definition}
\newtheorem{definition}{Definition}[]

\newcommand{\bb}{b}
\newcommand{\cc}{c}

\newcommand{\G}{\mathcal{G}}
\newcommand{\Gin}{\mathcal{G}_{\infty}}
\newcommand{\origin}{\textbf{\textit{o}}}

\newcommand{\R}{\mathbb{R}}  

\begin{document}


\title{Ricci flow of warped Berger metrics on $\mathbb{R}^{4}$}


\author{Francesco Di Giovanni}
\address{Department of Mathematics, University College London, Gower Street, London, WC1E 6BT, United Kingdom} 
\email{francesco.giovanni.17@ucl.ac.uk}




\begin{abstract}
We study the Ricci flow on $\R^{4}$ starting at an SU(2)-cohomogeneity 1 metric $g_{0}$ whose restriction to any hypersphere is a Berger metric. We prove that if $g_{0}$ has no necks and is bounded by a cylinder, then the solution develops a global Type-II singularity and converges to the Bryant soliton when suitably dilated at the origin. This is the first example in dimension $n > 3$ of a non-rotationally symmetric Type-II flow converging to a rotationally symmetric singularity model. Next, we show that if instead $g_{0}$ has no necks, its curvature decays and the Hopf fibres are not collapsed, then the solution is immortal. Finally, we prove that if the flow is Type-I, then there exist minimal 3-spheres for times close to the maximal time.

\end{abstract}

\maketitle


\section{Introduction}
Given a smooth Riemannian manifold $(M,g_{0})$, Hamilton's Ricci flow starting at $g_{0}$ is defined to be the geometric heat-type evolution equation \cite{threemanifolds}
\[
\frac{\partial g}{\partial t} = -2\text{Ric}_{g(t)}, \,\,\,\,\,\,\, g(0) = g_{0}.
\]
\noindent Shi proved that if $(M,g_{0})$ is complete and has bounded curvature, then the Ricci flow problem admits a solution \cite{shi}. Moreover, such solution is unique in the class of complete solutions with bounded curvature by the work of Chen and Zhu \cite{uniqueness}. A solution to the Ricci flow encounters a finite-time singularity at some $T < \infty$ if and only if \cite{formationsingularities}, \cite{shi}
\[
\limsup_{t\nearrow T}\,\sup_{M}\lvert\text{Rm}_{g(t)}\rvert_{g(t)} = \infty.
\]
\noindent Finite time singularities of the Ricci flow are classified as follows \cite{formationsingularities}: 
\begin{align*}
\emph{Type-I}&:\,\,\,\,\,\limsup_{t\nearrow T}\,(T-t)\,\sup_{M}\,\lvert \text{Rm}_{g(t)}\rvert_{g(t)} < \infty, \\
\emph{Type-II}&:\,\,\,\,\,\limsup_{t\nearrow T}\,(T-t)\,\sup_{M}\,\lvert \text{Rm}_{g(t)}\rvert_{g(t)} = \infty.
\end{align*}
\noindent According to results of Naber \cite{naber} and Enders-M\"uller-Topping \cite{type1} any parabolic dilation of a Type-I Ricci flow at a singular point converges to a non-flat gradient shrinking soliton. On the other hand, far less is known about Type-II singularities. 

The first examples of Type-II singularities in dimension $n\geq 3$ were found in \cite{type2} by Gu and Zhu, who considered a family of rotationally invariant Ricci flows on $S^{n}$. Angenent, Isenberg and Knopf later 
discovered Type-II spherically symmetric Ricci flows on $S^{n}$ that are modelled on degenerate neckpinches \cite{degenerateneck}.
Type-II singularities were also derived for rotationally invariant Ricci flows on $\R^{n}$ by Wu in \cite{wu} and later, for a larger set of initial data, by the author in \cite{work}.

Only very recently the first explicit examples of \emph{non} rotationally symmetric Type-II Ricci flows in dimension higher than three have been analysed by Appleton in \cite{appleton2} and by Stolarski in \cite{stolarski}, where they both obtained Ricci flat singularity models.

In our first result we show that a large family of 4-dimensional cohomogeneity 1 Ricci flows develop Type-II singularities modelled on the Bryant soliton \cite{bryant}. The Ricci flow on 4-dimensional cohomogeneity 1 manifolds has been recently studied on various topologies \cite{bettiol}, \cite{IKS1}, \cite{IKS2},\cite{appleton2}. In \cite{IKS1} Isenberg, Knopf and {\v{S}}e{\v{s}}um showed that the Ricci flow starting at a family of \emph{generalized warped Berger metrics} on $S^{1}\times S^{3}$ is Type-I and becomes rotationally symmetric around any singularity. This behaviour is regarded as a Type-I example of \emph{symmetry enhancement} along the Ricci flow.

In this work we study the Ricci flow evolving from a generalized warped Berger metric on $\R^{4}$. Namely, consider a metric $g_{0}$ invariant under the cohomogeneity 1 left-action of SU(2) on $\R^{4} = \mathbb{C}^{2}$. We can then write $g_{0}$ in Bianchi IX form as \cite{gibbons} 
\[
g_{0} = (ds)^{2} + a^{2}(s)\,\sigma_{1}\otimes \sigma_{1} + \bb^{2}(s)\,\sigma_{2}\otimes\sigma_{2} + \cc^{2}(s)\,\sigma_{3}\otimes\sigma_{3},
\]   
\noindent where $\{\sigma_{i}\}_{i=1}^{3}$ is a coframe dual to some Milnor frame $\{X_{i}\}_{i=1}^{3}$ on SU(2), with $X_{3}$ tangent to the Hopf fibres. We further assume that $g_{0}$ is invariant under rotations of the Hopf fibres. The last condition means that the left-invariant vector field $X_{3}$ is Killing thus extending 
the Lie algebra of $g_{0}$-Killing vectors to $\mathfrak{u}(2)$. In particular, we can write $g_{0}$ as 
\[
g_{0} = (ds)^{2} + \bb^{2}(s)\,(\sigma_{1}\otimes \sigma_{1} + \sigma_{2}\otimes\sigma_{2}) + \cc^{2}(s)\,\sigma_{3}\otimes\sigma_{3}.
\]
\noindent In analogy with \cite{IKS1} we finally assume that $\cc\leq \bb$ so that each non-degenerate fiber $\{s\}\times S^{3}$ is a Berger sphere. We call such metric a \emph{warped Berger} metric on $\R^{4}$. 

We first focus on initial data with linear volume growth.
\begin{definition}\label{definitionGm}
We let $\G$ be the set of complete bounded curvature warped Berger metrics $g_{0}$ on $\R^{4}$ satisfying the following conditions:
\begin{itemize}
\item[(i)] $\bb_{s}\geq 0$, $H\geq 0$, where $H(r)$ is the \emph{mean curvature} of the centred Euclidean sphere of Euclidean radius $r$ with respect to $g_{0}$. 
\item[(ii)] $\sup_{p\in\R^{4}}\bb(p) < \infty$.
\end{itemize} 
\end{definition}

The control on the sign of $H$ amounts to ruling out the existence of necks \cite{exampleneck}. We also note that the condition in (i) is weaker than asking for both $\bb$ and $\cc$ to be monotone. We prove that any Ricci flow starting in $\G$ converges to the Bryant soliton once suitably rescaled. This provides Type-II examples of symmetry enhancement and constitutes the first explicit case in dimension higher than three of a non-conformally flat Type-II Ricci flow converging to a rotationally symmetric singularity model.

\begin{theorem}\label{maintheoremtype2} Let $(\R^{4},g(t))_{0\leq t < T}$ be the maximal complete, bounded curvature solution to the Ricci flow starting at some $g_{0}\in\G$. The solution develops a Type-II singularity at some $T< \infty$ which is modelled on the Bryant soliton once suitably dilated. 
\end{theorem}
Theorem \ref{maintheoremtype2} resembles an analogous result recently derived in \cite{appleton2}, where Appleton studied the Ricci flow on the blow-up of $\mathbb{C}^{2}/\mathbb{Z}_{k}$ starting from a subclass of U(2)-invariant metrics as above. In particular, he showed that when $k=2$ if the initial metric is bounded by a cylinder at infinity and both $\bb$ and $\cc$ are increasing and satisfy a differential inequality, then the flow is Type-II and converges to the Eguchi-Hanson metric once suitably dilated.

Theorem \ref{maintheoremtype2} characterizes the Type-II singularity only partially. Indeed, while the Type-II singularity is not isolated, being the Bryant soliton asymptotically cylindrical (see, e.g., \cite{brendle3}), in general there is no control on the blow-up sequence giving rise to a family of shrinking cylinders. Moreover, the symmetries and the lack of necks suggest that the curvature ought to become large locally around the singular orbit (see also \cite{appleton2}). Equivalently, one should detect the Bryant soliton when dilating the flow at the origin $\origin$ by suitable factors. In our next result we address these issues hence providing  a much clearer picture of the Type-II singularity developed by Ricci flows in $\G$. In the following statement $R_{g(t)}$ represents the \emph{scalar curvature} of the Ricci flow solution. 
\begin{theorem}\label{maintheoremtype2bis} Let $(\R^{4},g(t))_{0\leq t < T}$ be the maximal complete, bounded curvature solution to the Ricci flow starting at some $g_{0}\in\G$. Then the following conditions hold:
\begin{itemize}
\setlength\itemsep{0.5em}
\item[(i)]\emph{(The Bryant soliton appears at the origin.)} There exists $t_{j}\nearrow T$ such that the rescaled Ricci flows $(\R^{4},g_{j}(t),\origin)$ defined by $g_{j}(t)\doteq R_{g(t_{j})}(\origin)g(t_{j} + (R_{g(t_{j})}(\origin))^{-1}t)$ converge to the Bryant soliton in the Cheeger-Gromov sense.
\item[(ii)]\emph{(The singularity is global.)} For any $p\in\R^{4}$ we have
\[
\limsup_{t\nearrow T}\left(\lvert \emph{Rm}_{g(t)}\rvert_{g(t)}(p)\right) = \infty.
\]
\item[(iii)]\emph{(Type-I blow-up at infinity.)} For any $t_{j}\nearrow T$ there exist a sequence $\{p_{j}\}$ and $\alpha > 0$ such that $d_{g_{0}}(\origin,p_{j})\rightarrow \infty$, $(T-t_{j})R_{g(t_{j})}(p_{j})\leq \alpha$, and the rescaled Ricci flows $(\R^{4},g_{j}(t),p_{j})$ defined by $g_{j}(t)\doteq R_{g(t_{j})}(p_{j})g(t_{j} + (R_{g(t_{j})}(p_{j}))^{-1}t)$ converge to the self-similar shrinking cylinder in the Cheeger-Gromov sense.
\item[(iv)]\emph{(Classification of singularity models.)} Any non-trivial singularity model is isometric to either the self-similar shrinking cylinder or the Bryant soliton.
\end{itemize}
\end{theorem}

We note that as an immediate consequence of (i) the scalar curvature and the full curvature are comparable in certain regions up to the singular time. We also point out that the phenomenon of Type-II enhancement of symmetries along the Ricci flow is intrinsic to the classification of 3-dimensional $\kappa$-solutions obtained by Brendle in \cite{brendle}. We also note that item (iv) in Theorem \ref{maintheoremtype2bis} relies on the recent extension of Brendle's work to higher dimensions by Li and Zhang \cite{brendle2}.

\vspace{0.05in}

Next, we show that the long-time property is satisfied by a class of Berger metrics whose curvature decays at infinity. General long-time existence results on non-compact manifolds usually rely on controlling the sign of the curvature and the volume growth \cite{cabezas}. From a different perspective, similar conclusions may be achieved when the analysis is restricted to families of homogeneous Riemannian metrics \cite{lafuente}. In this case the behaviour of the flow for long times is also understood \cite{bohm}. 
Instead of assuming a transitive action of a Lie group, one may study cohomogeneity 1 manifolds. In this direction, Oliynyk and Woolgar proved that the Ricci flow on $\mathbb{R}^{n}$ starting at an asymptotically flat spherically symmetric metric without necks is immortal \cite{woolgar}. The author improved this result by allowing any decay of the curvature \cite{work}. 

In our setting we consider the following set, whose intersection with $\G$ is empty.
\begin{definition}\label{definitionGin} We let $\Gin$ be the set of complete warped Berger metrics $g$ on $\R^{4}$ with \emph{positive injectivity radius} and satisfying the following conditions:
\begin{itemize}
\item[(i)] $\bb_{s} \geq 0$, $H\geq 0$.
\item[(ii)] $\lvert\text{Rm}_{g}\rvert_{g}(s)\rightarrow 0$ as $s\rightarrow \infty$ and there exist $\mu > 0$ and $s_{0} > 0$ such that $\cc(s) \geq \mu$ for any $s\geq s_{0}$.
\end{itemize}
\end{definition}

We prove the following:

\begin{theorem}\label{mainimmortalresult}
Let $(\R^{4},g(t))_{0\leq t < T}$ be the maximal complete, bounded curvature solution to the Ricci flow evolving from some $g_{0}\in\Gin$. Then the solution is immortal.
\end{theorem}

The long-time property may in general fail if we omit the requirement on the monotonicity of $\bb$ and $H$. Indeed by the adaptation of \cite{exampleneck} to $\R^{n+1}$ obtained in \cite{work} we deduce that if $g_{0}$ is asymptotically flat with $\bb = \cc$ and $(\R^{4},g_{0})$ contains a neck which is sufficiently pinched (in a precise way), then the Ricci flow is Type-I. It therefore remains to address the relation between Type-I singularities and existence of minimal hyperspheres for Berger Ricci flows. We recall that Angenent and Knopf constructed the first examples of nondegenerate neckpinches by evolving rotationally invariant metrics on $S^{n}$ containing minimal (stable) hyperspheres \cite{exampleneck}. Later the link between Type-I singularities and minimal spheres has been explored for Ricci flows on closed 3-manifolds by Song in \cite{song2}. In our setting we prove the following: 
\begin{theorem}\label{maintypeI}
Let $(\R^{4},g(t))_{0\leq t < T}$ be the maximal complete, bounded curvature solution to the Ricci flow evolving from a complete warped Berger metric $g_{0}$ with positive injectivity radius and curvature decaying at infinity. If $g(t)$ develops a Type-I singularity at $T < \infty$, then there exists $\delta > 0$ such that $(\R^{4},g(t))$ contains minimal embedded 3-spheres for any $t\in [T-\delta,T)$.
\end{theorem}

We may also apply Theorem \ref{maintheoremtype2} and Theorem \ref{mainimmortalresult} to derive two simple corollaries. First we immediately deduce that neither $\G$ nor $\Gin$ contain shrinking solitons. 
\begin{corollary}
There are no Taub-NUT like shrinking Ricci solitons on $\R^{4}$. 
\end{corollary}
In the second application we classify warped Berger Ricci flows with bounded nonnegative curvature. In particular we show that for positively curved warped Berger Ricci flows bounded by a cylinder at infinity, parabolic dilations at the origin along any sequence of times give rise to the Bryant soliton.
\begin{corollary}\label{mainpositiveresult}
Let $(\R^{4},g(t))_{0\leq t < T}$ be the maximal complete, bounded curvature Ricci flow starting at some complete warped Berger metric $g_{0}$ with bounded nonnegative curvature. Then $T$ is finite if and only if $\bb(\cdot,0)$ is bounded. If $T$ is finite, then for any $t_{j}\nearrow T$ the rescaled Ricci flows $(\R^{4},g_{j}(t),\origin)$ defined by $g_{j}(t)\doteq R_{g(t_{j})}(\origin)g(t_{j}+(R_{g(t_{j})}(\origin))^{-1}t)$ (sub)converge to the Bryant soliton. 
\end{corollary}
\subsection*{Outline.} We briefly describe the organization of the paper.
\\In Section 2 we discuss Berger metrics on $\R^{4}$ and we comment on the main assumptions. In Section 3 we show that the condition on the lack of necks persists along the Ricci flow. The main step consists in adapting the analogous argument adopted in \cite{appleton2}, which relies on the application of a general maximum principle for systems of parabolic equations. In Section 4 we study warped Berger Ricci flows evolving from initial data either in $\G$ or in $\Gin$. Similarly to \cite{IKS1} we show that the curvature is controlled by the size of the principal orbits and that the solution becomes rotationally symmetric around any singularity at some rate that breaks scale-invariance. An important ingredient, for the case of $\Gin$, is also given by the application of the Pseudolocality formula in \cite{pseudolocalityapplication} by Chau, Tam and Yu. In Section 5 we prove that any singularity model is rotationally symmetric by showing that the left-invariant Milnor frame diagonalizing the metric generates a copy of $\mathfrak{su}$(2) in the Lie algebra of Killing fields acting on the singularity model. We then apply the rigidity result obtained by Zhang in \cite{zhang} to classify these singularity models. In Section 6 we prove the main results. Theorem \ref{maintheoremtype2} heavily relies on the characterization of Type-I singularities in \cite{naber} and \cite{type1}. The appearance of the Bryant soliton follows from a result by Hamilton \cite{eternal} once we know that the singularity is Type-II. The localization of the Bryant soliton in (i) of Theorem \ref{maintheoremtype2bis} is a direct consequence of the convergence of left-invariant vector fields obtained in Section 5. The property that the singularity is global depends on the monotonicity assumption ($\bb_{s} \geq 0$, $H\geq 0$), which allows us to control the space-time region where the flow stays smooth. The Type-I  blow-up at infinity follows once we know that the solution becomes singular everywhere at some finite time $T$. We then obtain the classification of singularity models by combining the characterization of singularity models in Section 5 with the analysis in \cite{brendle2}. The proof of Theorem \ref{mainimmortalresult} follows from a contradiction argument. We show that if a Ricci flow in $\Gin$ develops a finite-time singularity, then \emph{any} singularity model is a non-compact $\kappa$-solution with Euclidean volume growth. However, in \cite{pseudolocality} Perelman showed that this is not possible. We finally address the proof of Theorem \ref{maintypeI}, which again depends on the characterization of Type-I Ricci flows obtained in \cite{type1}. The last section is devoted to deriving some easy applications of the main results. 

\subsection*{Acknowledgements.} The author would like to thank his advisor Jason Lotay for suggesting the problem, for the constant support and for many helpful conversations.
\section{Setting}
\subsection{Warped Berger metrics on $\R^{4}$.}
A compact Lie group $\mathsf{G}$ acting on a Riemannian manifold $(M,g)$ via isometries is said to act with \textit{cohomogeneity} 1 if the orbit space $M/\mathsf{G}$ is 1-dimensional.
If $M$ is a non-compact manifold, then the orbit space is either homeomorphic to $[0,1)$ or to $\mathbb{R},$ depending on whether there exists a singular orbit of codimension greater than one (see, e.g., \cite{cohomogeneity1}). We analyse the first case, with $\mathsf{G}$ and $M$ being $\text{SU(2)}$ and $\mathbb{R}^{4}$ respectively. 
\\Let us identify $S^{3}$ with SU(2) via the map $h:S^{3}\subset \mathbb{C}^{2}\rightarrow \text{SU(2)}$ defined in Euler coordinates by 
\[
(e^{i(\theta + \psi)}\cos(\phi),e^{i(\theta-\psi)}\sin(\phi))\mapsto \left[ {\begin{array}{cc}
   e^{i(\theta + \psi)}\cos(\phi) & -e^{-i(\theta-\psi)}\sin(\phi) \\
   e^{i(\theta-\psi)}\sin(\phi) & e^{-i(\theta +\psi)}\cos(\phi) \\
  \end{array} } \right],
\]
\noindent where $\phi\in [0,\pi/2), \psi\in [0,\pi), \theta\in [0,2\pi)$. By using the Maurer-Cartan formalism we find a basis of left-invariant 1-forms $\{\sigma_{i}\}$ given by (see also \cite{gibbons})
\begin{align*}
\sigma_{1} &= \sin(2\theta)d\phi - \sin(2\phi)\cos(2\theta)d\psi, \\ 
\sigma_{2} &=  \cos(2\theta)d\phi + \sin(2\phi)\sin(2\theta)d\psi, \\ 
\sigma_{3} &= \cos(2\phi)d\psi + d\theta,
\end{align*}
\noindent with dual left-invariant frame 
\begin{align}\label{leftinvariantframe}
X_{1} &= \sin(2\theta)\partial_{\phi} - \frac{\cos(2\theta)}{\sin(2\phi)}\partial_{\psi} + \cot(\phi)\cos(2\theta)\partial_{\theta}, \notag \\ 
X_{2} &=  \cos(2\theta)\partial_{\phi} + \frac{\sin(2\theta)}{\sin(2\phi)}\partial_{\psi} -\cot(2\phi)\sin(2\theta)\partial_{\theta}, \\ \notag
X_{3} &= \partial_{\theta}.
\end{align}
\noindent The vector fields $\{X_{i}\}$ satisfy the relations $[X_{i},X_{j}] = 2\epsilon_{ijk}X_{k}$, with $\epsilon_{ijk}$ the anti-symmetric symbol. Consider a left-invariant Riemannian metric $\bar{g}$ on $S^{3} = \text{SU(2)}$. Such metric can be diagonalized along some Milnor frame \cite[Chapter 1]{IRF}. Without loss of generality we may identify the Milnor frame with the left-invariant frame $\{X_{i}\}$ and write $\bar{g}$ as 
\[
\bar{g} \doteq \lambda_{1}^{2}\,\sigma_{1}\otimes\sigma_{1} + \lambda_{2}^{2}\,\sigma_{2}\otimes\sigma_{2} + \lambda_{3}^{3}\,\sigma_{3}\otimes\sigma_{3},
\]
\noindent for some positive constants $\{\lambda_{j}\}$. From the Koszul formula it follows that the left-invariant vector field $X_{3}$ is a $\bar{g}$-Killing vector if and only if $\lambda_{1} = \lambda_{2}$. In this case the metric $\bar{g}$ is also invariant under rotations of the Hopf-fibres and its Lie algebra of Killing vectors contains a copy of $\mathfrak{u}(2)$. We note that the choice $\lambda_{1} = \lambda_{2} = \lambda_{3} = 1$ recovers the round metric of constant curvature one while the choice $\lambda_{1} = \lambda_{2} = 1$ and $\lambda_{3} = \varepsilon < 1$ parametrizes the classic family of Berger spheres collapsing along the Hopf fibres as $\varepsilon \rightarrow 0$.
\\Let now $g$ be a Riemannian metric on $\R^{4}=\mathbb{C}^{2}$ invariant under the cohomogeneity 1 left-action of SU(2). The action admits one singular orbit consisting of the origin $\origin$ of $\R^{4}$. All the geometric information can then be obtained by restricting $g$ along a radial geodesic starting at $\origin$ and meeting the 3-hyperspheres orthogonally. Namely, on $\R^{4}\setminus\{\origin\}$ we have (see also \cite{gibbons})
\begin{align}\label{initialmetric}
\begin{split}
g &= \xi^{2}(x)dx\otimes dx + \bar{g}_{x} \\
&= \xi^{2}(x)dx\otimes dx + a^{2}(x)\,\sigma_{1}\otimes\sigma_{1} + \bb^{2}(x)\,\sigma_{2}\otimes\sigma_{2} + \cc^{2}(x)\,\sigma_{3}\otimes\sigma_{3},
\end{split}
\end{align}
\noindent where $\xi,a,b,c:(0,+\infty)\rightarrow (0,+\infty)$ are smooth radial functions. Since we are interested in SU(2)-invariant metrics on $\R^{4}$ whose restrictions to any hypersphere are Berger metrics we further require the metric $g$ to be invariant under rotations of the Hopf fibres. Equivalently, we assume that $a\equiv \bb$ in \eqref{initialmetric}. Moreover, we also restrict the analysis to those metrics satisfying the ordering constraint $\cc\leq\bb$.
For any radial coordinate $x > 0$ the metric 
\[
\bar{g}_{x} \doteq \bb^{2}(x)\,(\sigma_{1}\otimes\sigma_{1} + \sigma_{2}\otimes\sigma_{2}) + \cc^{2}(x)\,\sigma_{3}\otimes\sigma_{3}
\]
\noindent is then a left-invariant metric on the Euclidean hypersphere $S(\origin,x)$ with the $S^{1}$-fiber squashed by a factor $\cc(x)/\bb(x)\in (0,1]$. 
\\If we denote the $g$-distance from the origin by $s$, then we can write $g$ as 
\begin{equation}\label{initialmetricscoordinate}
g = ds\otimes ds + b^{2}(s)\,\left(\sigma_{1}\otimes \sigma_{1} + \sigma_{2}\otimes \sigma_{2}\right) + \cc^{2}(s)\,\sigma_{3}\otimes\sigma_{3}.
\end{equation}
\noindent We note that given a radial map $f$ on $\R^{4}$, then we interpret $f = f(s) = f(s(x))$ as a function of $x$ unless otherwise stated. We also have the relation  
\begin{equation}\label{changevariable}
\partial_{s} = \frac{1}{\xi(x)}\partial_{x}.
\end{equation} 
\noindent It is a general fact that $g$ in \eqref{initialmetricscoordinate} extends smoothly at the origin $\origin\in\R^{4}$ if and only if $\bb$ and $\cc$ extend to smooth odd functions at $x = 0$ and the following is satisfied:
\begin{equation}\label{smoothnessorigin}
\lim_{s\rightarrow 0}\frac{d \bb}{ds}(s)=\lim_{s\rightarrow 0}\frac{d \cc}{ds}(s) = 1. 
\end{equation}  
\noindent We note that the underlying topology plays a role in the analysis of the Ricci flow dynamics via the boundary conditions above. 
\\We point out that, as previously observed, the Lie algebra of Killing vectors for $g$ contains a copy of $\mathfrak{u}(2)$. Indeed, any U(2)-invariant metric on $\R^{4}=\mathbb{C}^{2}$ can be written as in \eqref{initialmetricscoordinate}, up to choosing a suitable Milnor frame (see, e.g., the case $k = 1$ in Section 2.2 of \cite{appleton2}). 
In analogy with \cite{IKS1} we refer to any (U(2)-invariant) metric on $\R^{4}$ of the form \eqref{initialmetricscoordinate} and satisfying $\cc\leq \bb$ as a (generalized) \emph{warped Berger} metric.
\subsection{Curvature terms.} Given a warped Berger metric $g_{0}$ a simple application of the Koszul formula (see also \cite[Appendix A]{IKS1}) allows to compute the sectional curvatures of the vertical planes
\begin{align}
k_{12} &= \frac{4\bb^2 - 3\cc^2}{\bb^4} - \frac{\bb_{s}^{2}}{\bb^{2}},\label{sectionalvertical12} \\
k_{13} = k_{23} &= \frac{\cc^2}{\bb^4} - \frac{\bb_{s}\cc_{s}}{\bb\cc}\, \label{sectionalvertical13} 
\end{align}
\noindent and of the mixed ones
\begin{align}\label{sectionalhorizontal01}
k_{01} = k_{02} &= -\frac{\bb_{ss}}{\bb}\, ,\\
k_{03} &= -\frac{\cc_{ss}}{\cc}\,. \label{sectionalhorizontal03}
\end{align}
\noindent We also note that we can write the scalar curvature as 
\begin{equation}\label{scalarcurvature}
R_{g} = 2(k_{01} + k_{02} + k_{03} + k_{12} + k_{13} + k_{23}).
\end{equation}
\subsection{Initial data for the Ricci flow.}
In this work we study the Ricci flow problem on $\R^{4}$ with initial condition given by a warped Berger metric $g_{0}$. We first assume that $g_{0}$ is bounded by a cylinder at infinity so that the Ricci flow evolving from $g_{0}$ always encounters a finite-time singularity. 
\\According to \cite{exampleneck}, \cite{work} if $(\R^{4},g_{0})$ contains necks, then the Ricci flow solution may be Type-I and converge to a shrinking cylinder once rescaled. In order to construct Type-II singularities we thus need to exclude these initial geometries. A generalization of the notion of neck discussed in \cite{exampleneck} to the SU(2)-invariant setting consists in considering whether the mean curvature of embedded hyperspheres changes sign. Namely, we introduce the quantity $H:x\rightarrow (2\bb_{s}/\bb + \cc_{s}/\cc)(x)$ representing the mean curvature of the centred Euclidean sphere of Euclidean radius $x$ with respect to $g_{0}$. We say that $g_{0}$ does not have necks when the mean curvature $H$ is nonnegative on $\R^{4}\setminus\{\origin\}$.
\\While in the rotationally symmetric setting a Sturmian type of argument guarantees that minimal hyperspheres cannot appear along the flow, one might expect that in the SU(2)-case the mean curvature could generally change sign along the flow. In order to prevent the latter phenomenon from happening, we require the spatial derivative $\bb_{s}$ to be nonnegative as well. 
\begin{definitionn}\label{definitionGm}
We let $\G$ be the set of complete \emph{bounded curvature} warped Berger metrics on $\R^{4}$ satisfying the following conditions:
\begin{itemize}
\item[(i)] $\bb_{s}\geq 0$, $H\geq 0$.
\item[(ii)] $\sup_{p\in\R^{4}}\bb(p) < \infty$.
\end{itemize} 
\end{definitionn}
\begin{remark}
From the formula for the mean curvature of the embedded hyperspheres we immediately derive that the assumption (i) in Definition \ref{definitionGm} is \emph{weaker} than asking for the monotonicity of both $\bb$ and $\cc$.
\end{remark}
In the second class of initial data for the Ricci flow we consider warped Berger metrics without necks but whose behaviour at infinity is not controlled by that of a cylinder. Namely, we require the curvature to decay to zero and the Hopf fibres to be not collapsed.
\begin{definitionn}\label{definitionGin} We let $\Gin$ be the set of complete warped Berger metrics $g$ on $\R^{4}$ with \emph{positive injectivity radius} and satisfying the following conditions:
\begin{itemize}
\item[(i)] $\bb_{s} \geq 0$, $H\geq 0$.
\item[(ii)] $\lvert\text{Rm}_{g}\rvert_{g}(s)\rightarrow 0$ as $s\rightarrow \infty$ and there exist $\mu > 0$ and $s_{0} > 0$ such that $\cc(s) \geq \mu$ for any $s\geq s_{0}$.
\end{itemize}
\end{definitionn}
\begin{remark}
We point out that the sets $\G$ and $\Gin$ are disjoint. For if $g_{0}\in\G\cap \Gin$, then $\bb < m$, for some $m > 0$, and hence by \eqref{sectionalvertical12} we find
\[
\lvert 4 - 3\frac{\cc^{2}}{\bb^{2}} - \bb_{s}^{2} \rvert \leq \lvert k_{12}\rvert \bb^{2} < m^{2} \lvert k_{12}\rvert. 
\]
\noindent Since the curvature is decaying to zero at infinity and $\cc\leq \bb$ we see that $\lvert\bb_{s}\rvert \geq 1/2$ outside some Euclidean ball $B(\origin,r)$, for $r$ large enough. Therefore $\bb(s)\rightarrow \infty$ and this is a contradiction. We conclude that if $g_{0}\in\Gin$, then $\bb(s)\rightarrow \infty$ being $\bb_{s}\geq 0$.
\end{remark}
\begin{remark} 
We observe that the well known Taub-NUT metric on $\R^{4}$ \cite{hawking} is a hyperk\"ahler metric belonging to $\Gin$ since the curvature decays to zero at cubic rate while both $\bb$ and $\cc$ are increasing.
\end{remark}

\subsection{The Ricci flow equations.}
If $g_{0}$ is a complete bounded curvature warped Berger metric on $\R^{4}$, then by \cite{shi} there exists a smooth complete solution to the Ricci flow problem. Such solution is unique among those complete solutions with bounded curvature \cite{uniqueness}. Therefore we have a well-defined notion of maximal time of existence for the Ricci flow solution. In the following we always let $(\R^{4},g(t))_{0\leq t < T}$ be the maximal complete, bounded curvature solution to the Ricci flow starting at some complete bounded curvature Berger metric $g_{0}$.    
\\ The Ricci flow diffeomorphism invariance and the uniqueness property ensure that the symmetries persist. Moreover, since the Ricci tensor is diagonal along the global frame $\{\partial_{x}, X_{1}, X_{2}, X_{3}\}$, the maximal Ricci flow solution starting at $g_{0}$ must be of the form 
\begin{align}\label{ricciflowsolution}
\begin{split}
g(t) &= \xi^{2}(x,t)\,dx\otimes dx + b^{2}(x,t)\,(\sigma_{1}\otimes\sigma_{1} + \sigma_{2}\otimes\sigma_{2}) + \cc^{2}(x,t)\,\sigma_{3}\otimes\sigma_{3} \\
&= ds\otimes ds + b^{2}(s(x),t)\,(\sigma_{1}\otimes\sigma_{1} + \sigma_{2}\otimes\sigma_{2}) + \cc^{2}(s(x),t)\,\sigma_{3}\otimes\sigma_{3},
\end{split}
\end{align}
\noindent where $s=s(x,t)$ is the $g(t)$-distance from the origin. In terms of the variables $s$ and $t$ the Ricci flow equations can be written as
\begin{align}\label{Ricciflowpdes1}
\bb_{t} &= \bb_{ss} + \left(\frac{\cc_{s}}{\cc} + \frac{\bb_{s}}{\bb} \right)\bb_{s} +\frac{2(\cc^{2} - 2\bb^{2})}{\bb^{3}} \\
\cc_{t} &= \cc_{ss} + 2\frac{\bb_{s}}{\bb}\cc_{s} - \frac{2\cc^{3}}{\bb^{4}}. \label{Ricciflowpdes2}
\end{align}
\noindent The choice of a meaningful geometric coordinate $s$ provides us with a parabolic form of the Ricci flow equations. However, we get a non vanishing commutator between $\partial_{t}$ and $\partial_{s}$ given by
\begin{equation}\label{commutatorformula}
\left[ \frac{\partial }{\partial t}, \frac{\partial}{\partial s} \right]    = -(\text{ln}(\xi))_{t}\frac{\partial}{\partial s} = - \left(2\frac{\bb_{ss}}{\bb} + \frac{\cc_{ss}}{\cc} \right)\frac{\partial}{\partial s}.
\end{equation}
\noindent We also report the formula for the (time-dependent) Laplacian along the Ricci flow. For any smooth function $f\in C^{\infty}(\mathbb{R}^{4})$ we have
\begin{equation}\label{formulalaplacian}
\Delta f \equiv f_{ss} + \left (2\frac{\bb_{s}}{\bb} + \frac{\cc_{s}}{\cc} \right )f_{s}.
\end{equation}
We dedicate the end of this subsection to proving that the Ricci flow solution $g(t)$ starting at $g_{0}$ remains a warped Berger metric until its maximal time of existence $T\leq \infty$. More precisely, we show the following:
\\
\begin{lemma}\label{consistencyRFassumption1}
Let $(\R^{4},g(t))_{0\leq t < T}$ be the maximal solution to the Ricci flow starting at some complete bounded curvature warped Berger metric $g_{0}$ and let $\varepsilon\doteq \inf_{x\geq 0}\cc/\bb(x,0)$. Then for any $(p,t)\in\R^{4}\times [0,T)$ we have
\[
\varepsilon \leq \frac{\cc}{\bb}(p,t)\leq 1.
\]
\end{lemma}
\begin{proof}
We first verify that the ordering $\cc\leq\bb$ is preserved along the flow for any $t\in[0,T)$. By \cite{shi} the curvature is bounded at any time slice $\R^{4}\times \{t\}$, with $t\in [0,T)$; thus from the Ricci flow equations we find that there exists some time dependent positive constant $\alpha(t)$ such that $\cc/\bb(\cdot,t)\leq \alpha(t) < \infty$ on $\R^{4}$.
As long as a (smooth) solution exists the boundary conditions \eqref{smoothnessorigin} are satisfied, which then imply that the function $f\doteq \log(\cc/\bb)$ is smoothly defined on $\R^{4}$ and equal to zero at the origin for any time. From the evolution equations \eqref{Ricciflowpdes1}, \eqref{Ricciflowpdes2} and the formula for the Laplacian \eqref{formulalaplacian} we get
\begin{equation}\label{evolutionlogcb}
f_{t} = \Delta f + \frac{4}{\bb^{2}}\left(1 - \frac{\cc^{2}}{\bb^{2}}\right).
\end{equation}
\noindent Therefore whenever $\cc/\bb > 1$ we find
\[
f_{t} < \Delta f.
\]
\noindent We can then apply the maximum principle \cite[Corollary 7.45]{hamiltonricciflow} and conclude that since $\cc/\bb(\cdot,0)\leq 1$, the same ordering persists along the flow. In fact, once we know that $\cc\leq \bb$ is preserved in time, a standard application of the strong maximum principle shows that if $c = b$ at some $(p_{0},t_{0})\in \R^{4}\times (0,T)$, then $\cc = \bb$ in a space-time neighbourhood of the point and thus $\cc=\bb$ everywhere for all earlier times by real analyticity of solutions to the Ricci flow \cite{bando}.
\\We now let $\varepsilon\in [0,1)$ be defined as in the statement. If $\varepsilon = 0$ there is nothing to show; we can then take $\varepsilon > 0$. Again from \cite{shi} it follows that $\cc/\bb(\cdot,t)\geq \alpha(t) > 0$; if we define $f\doteq \log(\varepsilon^{-1}\cc/\bb)$, since we have just shown that $\cc\leq\bb$ along the solution, we obtain
\[
f_{t} = \Delta f + \frac{4}{\bb^{2}}\left(1 - \frac{\cc^{2}}{\bb^{2}}\right) \geq \Delta f.
\]
\noindent We can apply the maximum principle and conclude that $\cc(\cdot,t)\geq \varepsilon\bb(\cdot,t)$ for any $t\in [0,T)$.
\end{proof}
\section{Ricci flow without necks}
In this section we show that the monotonicity assumptions $\bb_{s}\geq 0$ and $H\geq 0$ are preserved along the Ricci flow solution. The main ingredient is given by a maximum principle for systems of parabolic equations \cite[Theorem 13, p. 190]{protter} that recently Appleton used to derive similar conclusions for a family of $U(2)$-invariant Ricci flows with cylindrical asymptotics \cite{appleton2}. In the following we mainly adapt the argument in \cite{appleton2} to the topology of $\R^{4}$, i.e. to the boundary conditions given in \eqref{smoothnessorigin}. 
\subsection{Basic estimates.} Let $g_{0}$ be a complete bounded curvature Berger metric on $\R^{4}$. We collect a few preliminary bounds that are necessary to apply the maximum principle for systems to the evolution equations of $\cc\bb_{s}/\bb$ and $\cc H$. 
\begin{lemma}\label{basicestimate1}
For any $x_{0} > 0$ there exists $\delta > 0$ such that $\bb(x)\geq \delta > 0$ for all $x\geq x_{0}$.
\end{lemma}
\begin{proof}
By assumption there exists $\alpha > \sup_{\R^{4}}\lvert\text{Rm}_{g_{0}}\rvert_{g_{0}}$. Suppose for a contradiction that there exists $x_{0}$ such that $\bb(x) \leq \delta$ for any $x\geq x_{0}$, with $\delta^{2}\alpha < 1/2$. From \eqref{sectionalvertical12} we derive
\[
\lvert 4 - 3\frac{\cc^{2}}{\bb^{2}} - \bb_{s}^{2}\rvert \leq \alpha \bb^{2} \leq \frac{1}{2}
\]
\noindent for all $x\geq x_{0}$. Since $\cc \leq \bb$, we see that $\bb_{s}^{2}(x)\geq 1/2$ for any $x\geq x_{0}$ which contradicts the fact that $\bb$ is bounded. Therefore there exists a sequence of points $p_{j}\rightarrow \infty$ such that $\bb(p_{j}) > \delta$, with $\delta$ given above. Assume that there exists a sequence $q_{j}\rightarrow \infty$ such that $\bb(q_{j}) \leq \delta$. It follows that there exists a sequence of minima $\tilde{q}_{j}\rightarrow \infty$ such that $\bb(\tilde{q}_{j})\leq \delta$. From \eqref{sectionalvertical12} we get
\[
\lvert 4 - 3\frac{\cc^{2}}{\bb^{2}} - \bb_{s}^{2}\rvert (\tilde{q}_{j}) \equiv \lvert 4 - 3\frac{\cc^{2}}{\bb^{2}}\rvert (\tilde{q}_{j}) \leq \alpha\bb^{2}(\tilde{q}_{j})\leq \frac{1}{2},
\] 
\noindent which is not possible. The proof is then complete.
\end{proof}
A simple consequence of the previous Lemma is the following
\begin{corollaryy}\label{bbs/bbbounded}
Given $x_{0} > 0$ there exists $\alpha > 0$ such that 
\[
\sup_{\R^{4}\setminus B(\origin,x_{0})}\left \vert \frac{\bb_{s}}{\bb}\right \vert \leq \alpha.
\]
\end{corollaryy}
\begin{proof}
From \eqref{sectionalvertical12} we derive
\[
\frac{\bb_{s}^{2}}{\bb^{2}} \leq \frac{4\bb^{2} - 3\cc^{2}}{\bb^{4}} + \alpha.
\]
\noindent Given $x_{0} > 0$ we may apply Lemma \ref{basicestimate1} and conclude the proof.
\end{proof}
We also need to check that both $\bb_{s}$ and $\cc_{s}$ are exponentially bounded at spatial infinity.
\begin{lemma}\label{derivativesexponentiallybounded}
There exist $M > 0$ and $\alpha > 0$ such that 
\[ 
\lvert \cc_{s} \rvert + \lvert \bb_{s} \rvert \leq M\exp(\alpha s).
\]
\end{lemma}
\begin{proof}
According to \eqref{sectionalhorizontal01} and the uniform bound on the curvature we see that $\bb$ and hence $\bb_{ss}$ are exponentially bounded; thus the same holds for $\bb_{s}$ by integrating $\bb_{ss}$. Similar conclusions are satisfied by $\cc_{s}$.
\end{proof}
Finally a bound similar to Corollary \ref{bbs/bbbounded} is satisfied by $\cc_{s}/\cc$ as well. 
\begin{lemma}\label{ccsccbounded}
Given $x_{0} > 0$ there exists $\alpha > 0$ such that 
\[
\sup_{\R^{4}\setminus B(\origin,x_{0})}\left \vert \frac{\cc_{s}}{\cc}\right \vert \leq \alpha.
\]
\end{lemma}
\begin{proof}\cite[Lemma 3.4]{appleton2}.
\end{proof}
\subsection{Maximum principle for systems.} We consider the maximal Ricci flow solution $(\R^{4},g(t))_{0\leq t < T}$ evolving from a complete bounded curvature warped Berger metric $g_{0}$. We note that given $t_{0} < T$ then the estimates above hold uniformly for any $t\in [0,t_{0}]$ being the curvature uniformly bounded in the space-time region $\R^{4}\times [0,t_{0}]$. From the evolution equations \eqref{Ricciflowpdes1}, \eqref{Ricciflowpdes2}, the commutator formula \eqref{commutatorformula} and the expression for the mean curvature of embedded hyperspheres $H:(x,t)\rightarrow (2\bb_{s}/\bb + \cc_{s}/\cc)(x,t)$, we compute 
\begin{align}\label{evolutioncbbbs}
\left(\frac{\cc}{\bb}\bb_{s}\right)_{t} &= \left(\frac{\cc}{\bb}\bb_{s}\right)_{ss} + \left(\frac{\cc}{\bb}\bb_{s}\right)_{s}\left(2\frac{\bb_{s}}{\bb}- \frac{\cc_{s}}{\cc}\right) + \frac{1}{\bb^{2}}\left(\frac{\cc}{\bb}\bb_{s}\right)\left(8 - 10\frac{\cc^{2}}{\bb^{2}} - 2\bb_{s}^{2}\right)  + 4\frac{\cc^{2}}{\bb^{4}}\cc_{s}, \\ (\cc H)_{t} &= (\cc H)_{ss} + (cH)_{s} \left(2\frac{\bb_{s}}{\bb}- \frac{\cc_{s}}{\cc}\right) + 2\frac{\cc H}{\bb^{2}}\left(\frac{\cc^{2}}{\bb^{2}} - \bb_{s}^{2}\right) + \frac{16}{\bb^{2}}\left(\frac{\cc}{\bb}\bb_{s}\right)\left(1 - \frac{\cc^{2}}{\bb^{2}}\right). \label{evolutioncH}
\end{align}
\noindent We may now prove the main result of this section.
\begin{lemma}\label{cruciallemma}
Let $(\R^{4},g(t))_{0\leq t < T}$ be the maximal Ricci flow solution starting at a complete bounded curvature warped Berger metric $g_{0}$. If $(\cc/\bb)\bb_{s}(\cdot,0) \geq 0$ and $\cc H(\cdot,0)\geq 0$ then $(\cc/\bb)\bb_{s}(\cdot,t) > 0$ and $\cc H(\cdot,t) > 0$ for any $t\in(0,T)$.
\end{lemma}
\begin{proof}
Suppose that there exist $x_{0}$ and $t_{0} > 0$ such that $(\cc/\bb)\bb_{s}(x_{0},t_{0}) = - z < 0$, for some $z > 0$. By the boundary conditions $x_{0} > 0$ and there exists $\delta = \delta (t_{0}) > 0$ such that 
\[
\inf_{B(\origin,\delta)\times [0,t_{0}]}\frac{\cc}{\bb}\bb_{s}(x,t)\geq \frac{1}{2}, \,\,\,\,\,\,\,\,\,\, \inf_{B(\origin,\delta)\times [0,t_{0}]}(\cc H)(x,t)\geq \frac{1}{2}.
\]
\noindent Using the commutator formula \eqref{commutatorformula} we may rewrite the evolution equations \eqref{evolutioncbbbs} and \eqref{evolutioncH} in the space-time region $(\R^{4}\setminus B(\origin,\delta))\times [0,t_{0}]$ as 
\begin{align*}
\left(\frac{\cc}{\bb}\bb_{s}\right)_{t} &= \frac{1}{\xi^{2}}\left(\frac{\cc}{\bb}\bb_{s}\right)_{xx} + \frac{1}{\xi}\left(2\frac{\bb_{s}}{\bb} - \frac{\cc_{s}}{\cc} - \frac{\xi_{x}}{\xi^{2}}\right)\left(\frac{\cc}{\bb}\bb_{s}\right)_{x} \\ &+ \frac{1}{\bb^{2}}\left(\frac{\cc}{\bb}\bb_{s}\right)\left(8 - 18\frac{\cc^{2}}{\bb^{2}} - 2\bb_{s}^{2}\right)  + 4\frac{\cc^{2}}{\bb^{4}}(\cc H) 
\end{align*}
\noindent and 
\begin{align*}
(\cc H)_{t} &= \frac{1}{\xi^{2}}\left(\cc H\right)_{xx} + \frac{1}{\xi}\left(2\frac{\bb_{s}}{\bb} - \frac{\cc_{s}}{\cc} - \frac{\xi_{x}}{\xi^{2}}\right)\left(\cc H\right)_{x} \\ &+ \frac{16}{\bb^{2}}\left(\frac{\cc}{\bb}\bb_{s}\right)\left(1 - \frac{\cc^{2}}{\bb^{2}}\right) + \frac{2}{\bb^{2}}(\cc H)\left(\frac{\cc^{2}}{\bb^{2}} - \bb_{s}^{2}\right).
\end{align*}
\noindent From Lemma \ref{basicestimate1} and Corollary \ref{bbs/bbbounded} we derive that the zero order coefficients are uniformly bounded in $(\R^{4}\setminus B(\origin,\delta))\times [0,t_{0}]$. Moreover, by Lemma \ref{consistencyRFassumption1} we know that the ordering $\cc\leq \bb$ is preserved along the flow, therefore the coupling coefficients $4\cc^{2}/\bb^{4}$ and $16/\bb^{2}(1 - \cc^{2}/\bb^{2})$ are both nonnegative. Similarly to \cite{appleton2} we can introduce a barrier function 
\[
W: (x,t)\rightarrow \exp\left(\frac{s^{2}(x,t)}{1 - \beta t} + \lambda t\right)
\]
\noindent for $t\leq \min\{t_{0},(2\beta)^{-1}\}$ and compute the evolution equations of $\cc\bb_{s}/\bb + \epsilon W$ and $\cc H + \epsilon W$ for any $\epsilon > 0$. Using Corollary \ref{bbs/bbbounded}, Lemma \ref{ccsccbounded} and standard distortion estimates of the distance function it is straightforward to check that there exist $\beta = \beta(t_{0})$ and $\lambda = \lambda(t_{0})$ such that 
\begin{align*}
\left(\frac{\cc}{\bb}\bb_{s} + \epsilon W\right)_{t} &> \frac{1}{\xi^{2}}\left(\frac{\cc}{\bb}\bb_{s} + \epsilon W\right)_{xx} + \frac{1}{\xi}\left(2\frac{\bb_{s}}{\bb} - \frac{\cc_{s}}{\cc} - \frac{\xi_{x}}{\xi^{2}}\right)\left(\frac{\cc}{\bb}\bb_{s} + \epsilon W\right)_{x} \\ &+ \frac{1}{\bb^{2}}\left(\frac{\cc}{\bb}\bb_{s} + \epsilon W\right)\left(8 - 18\frac{\cc^{2}}{\bb^{2}} - 2\bb_{s}^{2}\right)  + 4\frac{\cc^{2}}{\bb^{4}}(\cc H + \epsilon W)
\end{align*}
\noindent and similarly for the evolution equation of $\cc H + \epsilon W$. By assumption $\cc\bb_{s}/\bb + \epsilon W(\cdot,0) > 0$ and $\cc\bb_{s}/\bb + \epsilon W(\delta,t) > 0$ for any $t\in [0,\min\{t_{0},(2\beta)^{-1}\}]$. Since by Lemma \ref{derivativesexponentiallybounded} $\cc\bb_{s}/\bb$ is exponentially bounded in space uniformly in the time interval $[0,t_{0}]$ we also get that $(\cc\bb_{s}/\bb + \epsilon W)(x,t)\rightarrow \infty$ as $x\rightarrow \infty$ for any $t\in [0,\min\{t_{0},(2\beta)^{-1}\}]$. The same conclusions are satisfied by $\cc H + \epsilon W$. We can then apply \cite[Theorem 13, p.190]{protter} and conclude that $\cc\bb_{s}/\bb$ and $\cc H$ stay nonnegative along the flow. The strict inequality in the statement then follows from using the maximum principle \cite[Theorem 3, p.38]{friedman} and the real analyticity of the Ricci flow solutions \cite{bando} once we know that $\bb_{s}$ and $H$ are nonnegative.
\end{proof}
\section{Analysis of the Ricci flow}
In this section we derive the main curvature estimates for the Ricci flow solution $(\R^{4},g(t))$ evolving from a warped Berger metric metric $g_{0}$. In the first part we focus on the case $g_{0}\in\G$. Similarly to the analyses in \cite{IKS1} and \cite{IKS2} (which are performed on $S^{1}\times S^{3}$ and $S^{2}\tilde{\times}S^{2}$ respectively) we prove that away from the origin the Ricci flow is controlled by the size of the principal orbits. In particular, we show that the formation of a singularity at some positive $x$ (i.e. along the Euclidean hypersphere of radius $x$) is equivalent to $\bb(x,t)$ converging to zero as $t \rightarrow T$. We also describe the behaviour of the flow as the time approaches $T$. Analogously to \cite{IKS1}, we prove that around any singularity the solution becomes rotationally symmetric at some rate that breaks scale-invariance. 
\\In the second part we extend the previous estimates to Ricci flows starting at some $g_{0}\in\Gin$. Moreover, for this class of solutions we also prove that (a scale-invariant version of) the mean curvature of minimal hyperspheres admits a uniform positive lower bound in the compact region where singularities may form. 
\\For notational reasons we always let $\alpha$ denote a positive constant only depending on $g_{0}$ that may change from line to line.
\subsection{Curvature estimates in $\G$.}
Throughout this section we let $(\R^{4},g(t))_{0\leq t < T}$ be the maximal complete, bounded curvature Ricci flow solution evolving from some $g_{0}\in\G$. Since $\bb(\cdot,0)$ is bounded from above and $\cc\bb^{2}(\cdot,0)$ is increasing, because we have $(\cc\bb^{2})_{s} = \cc\bb^{2} H$, we deduce that there exists $\varepsilon > 0$ such that $\cc/\bb(\cdot,0) \geq \varepsilon$. By Lemma \ref{consistencyRFassumption1} we obtain 
\begin{equation}\label{uniformlowerboundforcb}
\varepsilon \leq \frac{\cc}{\bb}(\cdot,t) \leq 1,
\end{equation}
\noindent uniformly in the space-time $\R^{4}\times [0,T)$. We observe that \eqref{uniformlowerboundforcb} is not available for the topologies analysed in \cite{IKS2} and \cite{appleton2}.
\\ Next, we show that the maximal time of existence $T$ is finite.
\begin{lemma}\label{finite-timesingularity}
Let $(\R^{4},g(t))_{0\leq t < T}$ be the Ricci flow solution starting at $g_{0}\in\G$. Then 
\[
\sup_{p\in \R^{4}}\bb(p,t) \leq \sup_{p\in\R^{4}}\bb(p,0)
\]
\noindent for any $t\in [0,T)$. Moreover, we have $T \leq \frac{\sup\bb^{2}(\cdot,0)}{4}$.
\end{lemma}
\begin{proof}
From the boundary conditions we deduce that $\bb^{2}(\cdot,t)$ is a smooth function on $\R^{4}$ as long as the solution exists. By \eqref{Ricciflowpdes1} and \eqref{formulalaplacian} we get 
\[
\partial_{t}\bb^{2} = \Delta \bb^{2} -4\bb_{s}^{2} + 4\frac{\cc^{2}}{\bb^{2}} - 8 \leq \Delta \bb^{2}  - 4.
\]
\noindent The conclusions then follow from the maximum principle \cite[Theorem 12.14]{ricciflowtechniques2}.
\end{proof}
\begin{remark}
From Lemma \ref{cruciallemma} and Lemma \ref{finite-timesingularity} we derive that the set $\G$ is preserved along the Ricci flow.
\end{remark}
Next, we prove that $\bb_{s}$ and $\cc_{s}$ are uniformly bounded in the space-time. The evolution equations of the first order spatial derivative are given by 
\begin{equation}\label{equationbs}
(\bb_{s})_{t} = \Delta (\bb_{s}) - 2\frac{\bb_{s}}{\bb}(\bb_{s})_{s} + \left (\frac{4}{\bb^{2}} - \frac{\bb_{s}^{2}}{\bb^{2}} - \frac{\cc_{s}^{2}}{\cc^{2}} -6\frac{\cc^{2}}{\bb^{4}} \right )\bb_{s} + 4\frac{\cc}{\bb^{3}}\cc_{s}
\end{equation}
\noindent and
\begin{equation}\label{equationcs}
(\cc_{s})_{t} = \Delta (\cc_{s}) - 2\frac{\cc_{s}}{\cc}(\cc_{s})_{s} - \left ( 6\frac{\cc^{2}}{\bb^{4}} + 2\frac{\bb_{s}^{2}}{\bb^{2}} \right )\cc_{s} + 8\frac{\cc^{3}}{\bb^{5}}\bb_{s}.
\end{equation}
\begin{lemma}\label{firstderivativesbounded}
There exists $\alpha > 0$ such that $\lvert \bb_{s} \rvert \leq \alpha$ and $\lvert \cc_{s} \rvert \leq \alpha$ in $\R^{4}\times [0,T)$.
\end{lemma} 
\begin{proof}
From Lemma \ref{cruciallemma} and Lemma \ref{finite-timesingularity} we derive that $\bb_{s}(\cdot,t)$ is integrable for any $t\in [0,T)$. Moreover, by \cite{shi} we see that for any $t\in [0,T)$ there exists $\alpha(t) < \infty$ such that $\lvert \bb_{ss}/\bb\rvert \leq \alpha(t)$. Since by Lemma \ref{finite-timesingularity} $\bb$ is uniformly bounded from above we deduce that $\bb_{s}(x,t)\rightarrow 0$ as $x\rightarrow \infty$ for any $t\in [0,T)$. Therefore, if $\bb_{s}$ becomes unbounded (from above) as $t\nearrow T$, then there exists a critical point $p_{0}$ for $\bb_{s}(\cdot, t_{0})$ where $\bb_{s} = \bar{\alpha}$ for the first time, for some $\bar{\alpha}$ large to be chosen below and for some $t_{0} > 0$. 
Evaluating \eqref{equationbs} at $(p_{0},t_{0})$ we get
\[
(\bb_{s})_{t}(p_{0},t_{0}) \leq \frac{1}{\bb^{2}}\left (4\bar{\alpha} - \bar{\alpha}^{3} - \bar{\alpha} \cc_{s}^{2} - 6\bar{\alpha}\left(\frac{\cc^{2}}{\bb^{2}}\right) + 4 \frac{\cc}{\bb}\cc_{s}\right ). 
\]
\noindent By choosing $\bar{\alpha}> \max\{\sup\lvert \bb_{s}\rvert(\cdot, 0), 2\}$ one easily checks that the $\cc_{s}-$ quadratic polynomial in the brackets does not admit roots, thus proving $(\bb_{s})_{t}(p_{0},t_{0}) < 0$. The exact same argument works for the lower bound of $\bb_{s}$. In fact, the lower bound for $\bb_{s}$ also follows from Lemma \ref{cruciallemma}.
\\We now adapt the argument for $\cc_{s}$. Since $\bb^{2}\cc H = (\bb^{2}\cc)_{s}$ and $\bb^{2}\cc$ is bounded from above, we see that $\bb^{2}\cc H(\cdot,t)$ is integrable. By differentiating we find
\[
(\bb^{2}\cc)_{ss} = 2\bb\cc\bb_{ss} + 2\bb_{s}^{2}\cc + 4\bb\bb_{s}\cc_{s} + \bb^{2}\cc_{ss}.
\]
\noindent From \eqref{sectionalvertical13} we derive $\lvert \bb_{s}\cc_{s}\rvert(\cdot,t) \leq \alpha(t)\bb\cc(\cdot,t) + \cc^{3}/\bb^{3}(\cdot,t) \leq \alpha(t)$ being the curvature bounded at any time slice $\R^{4}\times \{t\}$ for $t\in [0,T)$. Similarly $\lvert \cc_{ss}\rvert(\cdot,t)\leq \alpha(t)$. Therefore, since $\bb_{s}$ is uniformly bounded in the space-time we conclude that $\lvert(\bb^{2}\cc)_{ss}\rvert (\cdot,t)\leq \alpha(t)$, which implies $\bb^{2}\cc H(x,t)\rightarrow 0$ as $x\rightarrow \infty$ for any $t\in[0,T)$. In particular $\cc_{s}(x,t)\rightarrow 0$ as $x\rightarrow \infty$ because $\bb$ (and hence $\cc$) is uniformly bounded from above and $\bb_{s}(x,t)\rightarrow 0$. One can then argue as above that if $\cc_{s}$ becomes unbounded as $t\nearrow T$, then there exists a first maximum $p_{0}$ where $\cc_{s}$ attains a sufficiently large value $\bar{\alpha}$ at some $t_{0} > 0$ for the first time. It follows that
\[
(\cc_{s})_{t}(p_{0},t_{0}) \leq \frac{1}{\bb^{2}}\left(\bar{\alpha}\left(-6\frac{\cc^{2}}{\bb^{2}} - 2\bb_{s}^{2}\right) + 8\frac{\cc^{3}}{\bb^{3}}\bb_{s}\right).
\]
\noindent By Lemma \ref{consistencyRFassumption1} we know that the ordering $\cc\leq\bb$ is preserved. Therefore for $\bar{\alpha}$ large enough the right hand side is strictly negative. The same conclusion holds for the lower bound.
\end{proof}
From the previous Lemma and the condition $\cc \leq \bb$ we immediately derive the following bounds for the vertical sectional curvatures. From now on any estimate is satisfied in the space-time $\R^{4}\times [0,T)$ unless otherwise stated.
\begin{corollaryy}\label{verticalsectionalcontrolled}
There exists $\alpha> 0$ such that 
\[
\lvert k_{12} \rvert + \lvert k_{13}\rvert \leq \frac{\alpha}{\bb^{2}}.
\]
\end{corollaryy}
The following estimate is a necessary step to prove that the solution to the Ricci flow becomes spherically symmetric at any singularity forming away from the origin.
\begin{lemma}\label{bssquareminus4}
The following holds as long as the solution exists:
\[
\sup_{\R^{4}}\left(\frac{1}{\bb}\left(\bb_{s}^{2}-4\right)\right)_{+}(\cdot,t) \leq \sup_{\R^{4}}\left(\frac{1}{\bb}\left(\bb_{s}^{2}-4\right)\right)_{+}(\cdot,0).
\]
\end{lemma}
\begin{proof}
Let us denote the quantity $(\bb_{s}^{2}-4)/\bb$ by $\varphi$. By the boundary conditions $\varphi$ is uniformly bounded from above as $x\rightarrow 0$. Moreover, as we have already argued in the proof of Lemma \ref{firstderivativesbounded}, we find that $\varphi(x,t)$ becomes negative for $x$ large enough. We may then let $(p_{0},t_{0})$ be the maximum point among prior times where $\varphi$ attains some positive value $\alpha$. A direct computation gives
\[
\varphi_{t} = \Delta \varphi - 2\frac{\bb_{ss}^{2}}{\bb} - \frac{\varphi^{2}}{\bb} - \varphi \left(2\frac{\bb_{s}^{2}}{\bb^{2}} + 2\frac{\cc^{2}}{\bb^{4}} \right) - 12 \frac{\cc^{2}}{\bb^{5}}\bb_{s}^{2} + 2\cc_{s}\bb_{s}\left(4\frac{\cc}{\bb^{4}} -\frac{\cc_{s}\bb_{s}}{\bb\cc^{2}}  \right).
\]
\noindent Evaluating the evolution equation at $(p_{0},t_{0})$ we get
\[\varphi_{t}(p_{0},t_{0}) \leq \frac{1}{\bb^{3}}\left(-\frac{7}{2}\bb_{s}^{4} + \bb_{s}^{2}\left(20 - 14\frac{\cc^{2}}{\bb^{2}} \right) - 2\frac{\bb_{s}^{2}\cc_{s}^{2}\bb^{2}}{\cc^{2}} + 8\frac{\bb_{s}\cc_{s}\cc}{\bb} - 24 + 8\frac{\cc^{2}}{\bb^{2}}\right).
\]
\noindent We now regard the term in the brackets as a quadratic polynomial in $\cc_{s}$. Chosen $\alpha > 0$, we can find $\epsilon > 0$ such that $\bb_{s}^{2} = 4 + \epsilon$. The discriminant of the polynomial is given by
\[
8\frac{\bb_{s}^{2}\bb^{2}}{\cc^{2}}\left (- 8 \epsilon - \frac{7}{2}\epsilon^{2} - 14\epsilon \frac{\cc^{2}}{\bb^{2}} - 48\frac{\cc^{2}}{\bb^{2}} + 8 \frac{\cc^{4}}{\bb^{4}}  \right) < 0,
\]
\noindent where we have again used that the ordering $\cc \leq \bb$ is preserved by Lemma \ref{consistencyRFassumption1}. Therefore, the quantity $\sup_{\R^{4}}\varphi_{+}(\cdot,t)$ is non-increasing along the solution.
\end{proof} 
In the next Lemma we prove that if $\cc(x,t)$ converges to zero as $t\rightarrow T$ (along some sequence of times) for some $x > 0$, then the metric becomes rotationally symmetric at $x$.
\begin{lemma}\label{rotationalsymmetry}
There exists $\alpha > 0$ such that 
\[
\varphi\doteq \frac{1}{\bb}\left (\frac{\bb}{\cc} - 1 \right) \leq \alpha.
\]
\end{lemma}
\begin{proof}
We first prove a useful characterization of the behaviour of the second order spatial derivatives at infinity.
\begin{claim}\label{secondderivativesdecay}
For any $t\in \left[\frac{T}{2},T\right)$ we have 
\[
\bb_{ss}(x,t) \xrightarrow[]{x\nearrow \infty} 0, \,\,\,\,\,\,\,\,\,\, \cc_{ss}(x,t)\xrightarrow[]{x\nearrow \infty} 0.
\]
\end{claim}
\begin{proof}[Proof of Claim \ref{secondderivativesdecay}]
From the proof of Lemma \ref{firstderivativesbounded} we see that $\bb_{s}\rightarrow 0$ at infinity, which implies that the integral of $\bb_{ss}(\cdot,t)$ has a finite limit for any $t\geq T/2.$ By Shi's derivative estimates and the Koszul formula we find that for any $t\in[T/2,T)$ there exists $\alpha(t) > 0$ such that
\[
\alpha (t) \geq \lvert \nabla \text{Rm}_{g(t)}(\partial_{s},\partial_{s},X_{1}/\bb,\partial_{s},X_{1}/\bb)\rvert = \lvert \partial_{s}k_{01} \rvert = \left \vert \partial_{s}\left(\frac{\bb_{ss}}{\bb}\right) \right \vert,
\] 
\noindent which then proves that $\lvert\bb_{sss}\rvert(\cdot,t)\leq \alpha(t)$ being both $\bb$ and $\bb_{s}$ uniformly bounded. Therefore $\bb_{ss}(x,t)\rightarrow 0$ as $x\rightarrow \infty$ for any $t\in [T/2,T)$. 
\\Again from the proof of Lemma \ref{firstderivativesbounded} we derive that the integral of $(\bb^{2}\cc)_{ss}(\cdot,t)$ is convergent for any $t\in [T/2,T)$. By computing the derivative $(\bb^{2}\cc)_{sss}$ and using again Shi's derivative estimates we obtain that $(\bb^{2}\cc)_{ss}\rightarrow 0$, which also implies $\cc_{ss}(x,t)\rightarrow 0$ as $x\rightarrow \infty$ for any $t\in [T/2,T)$.
\end{proof}
By the boundary conditions the function $\varphi = 1/\cc - 1/\bb$ is continuously defined at the origin and identically zero. From \eqref{uniformlowerboundforcb} we see that $\varphi$ is bounded along any time slice $\R^{4}\times \{t\}$, for $t\in [0,T)$. The evolution equation of $\varphi$ is given by
\begin{equation}\label{timederivativeuseful}
\varphi_{t} = \Delta \varphi + \frac{\bb_{s}^{2}}{\bb^{3}} - \frac{\cc_{s}^{2}}{\cc^{3}} + \frac{1}{\bb^{3}}\left(-4 + 2 \frac{\cc}{\bb} + 2 \frac{\cc^{2}}{\bb^{2}} \right).
\end{equation}
\noindent By Lemma \ref{firstderivativesbounded} and Claim \ref{secondderivativesdecay} we see that for any $\delta > 0$ and $t\in [T/2,T)$ there exists $x_{0} = x_{0}(\delta,t)$ such that the time derivative of $\varphi$ can be bounded for $x$ larger than $x_{0}$ as 
\[
\varphi_{t} \leq \delta + \frac{1}{\bb^{3}}\left(-4 + 2\frac{\cc}{\bb} + 2\frac{\cc^{2}}{\bb^{2}} \right) \leq \delta - \frac{2\cc}{\bb^{3}}\varphi \leq \delta -\frac{2\varepsilon}{\bb^{2}}\varphi \leq \delta- \eta\varphi,
\]
\noindent for some $\eta > 0$, where we have used Lemma \ref{consistencyRFassumption1} and Lemma \ref{finite-timesingularity}. Therefore, if $\varphi$ does not stay bounded as $t\nearrow T$, then there exists a sequence of maxima diverging as the solution approaches its maximal time of existence. 
\\We introduce the quantity $\bar{\alpha} \doteq  L\varepsilon^{-1}(\varepsilon^{-1} - 1)$, where $L = \sup_{\R^{4}}(\bb_{s}^{2}-4)/\bb(\cdot,0)$ and $\varepsilon$ is chosen to satisfy Lemma \ref{consistencyRFassumption1}. Suppose that $(p_{0},t_{0})$ is a space-time maximum point among prior times where $\varphi$ attains some value greater than $\bar{\alpha}$. By evaluating \eqref{timederivativeuseful} at $(p_{0},t_{0})$ we see that
\[
\varphi_{t}(p_{0},t_{0}) \leq \frac{1}{\bb^{3}}\left(\bb_{s}^{2}\left(1 - \frac{\cc}{\bb} \right) - 4 + 2\frac{\cc}{\bb} + 2 \frac{\cc^{2}}{\bb^{2}} \right).
\]
\noindent Using Lemma \ref{bssquareminus4}, we can estimate the time derivative from above as
\begin{align*}
\varphi_{t}(p_{0},t_{0}) &\leq \frac{1}{\bb^{3}}\left((L\bb + 4)\left(1 - \frac{\cc}{\bb} \right) - 4 + 2\frac{\cc}{\bb} + 2 \frac{\cc^{2}}{\bb^{2}} \right) \\ 
&= \frac{1}{\bb^{3}}\left(- 2\frac{\cc}{\bb} + 2 \frac{\cc^{2}}{\bb^{2}} + L(\bb - \cc)\right) \\
&= \frac{\cc}{\bb^{3}}\varphi\left(- 2\frac{\cc}{\bb} + L\bb \right) \leq \frac{\cc}{\bb^{3}}\varphi\left(-2\varepsilon + L\bb \right).
\end{align*}
\noindent From the definition of $\varphi$ we derive 
\[
\bb \leq \frac{1}{\varphi}\left(\varepsilon^{-1} - 1\right ),
\]
\noindent which then yields
\[
\varphi_{t}(p_{0},t_{0}) \leq \frac{\cc}{\bb^{3}}\varphi\left(-2\varepsilon + L\frac{1}{\varphi}\left(\varepsilon^{-1} - 1\right ) \right) \leq \frac{\cc}{\bb^{3}}\varphi\left(-2\varepsilon + \varepsilon \right) < 0.
\]
\end{proof}
An analogous bound holds for the first order spatial derivatives. Namely, we have the following
\begin{lemma}\label{rotationalsymmetryorder1}
There exists $\alpha > 0$ such that
\[
\left \vert \frac{\cc_{s}}{\cc} - \frac{\bb_{s}}{\bb}\right \vert \leq \alpha.
\]
\end{lemma}
\begin{proof}
Define $\psi\doteq \cc_{s}/\cc - \bb_{s}/\bb$. The function $\psi$ extends to zero at the origin due to the boundary conditions. As argued before $\psi(x,t)\rightarrow 0$ as $x\rightarrow \infty$ for any $t\in [0,T)$. 
Consider the upper bound. Suppose that there exists a large value $\bar{\alpha}$ which $\psi$ attains for the first time at some maximum space-time point $(p_{0},t_{0})$. The evolution equation of $\psi$ is 
\[
\psi_{t} = \Delta \psi -\psi\left(\frac{\cc_{s}^{2}}{\cc^{2}} + 8 \frac{\cc^{2}}{\bb^{4}} + 2\frac{\bb_{s}^{2}}{\bb^{2}}\right) - 8\frac{\bb_{s}}{\bb^{3}}\left(1 - \frac{\cc^{2}}{\bb^{2}} \right).
\]
\noindent We can then evaluate both sides at $(p_{0},t_{0})$ and use Lemma \ref{rotationalsymmetry} to get
\begin{align*}
\psi_{t}(p_{0},t_{0}) &\leq \frac{1}{\bb^{2}}\left(-\bar{\alpha}\left(\frac{\bb^{2}\cc_{s}^{2}}{\cc^{2}} + 8 \frac{\cc^{2}}{\bb^{2}} + 2\bb_{s}^{2} \right) + 8\frac{\cc \lvert\bb_{s}\rvert}{\bb}\left (1 + \frac{\cc}{\bb}\right)\left(\frac{1}{\cc}-\frac{1}{\bb}\right) \right) \\ &\leq \frac{8}{\bb^{2}}\left(-\bar{\alpha}\frac{\cc^{2}}{\bb^{2}} + 2\alpha\lvert\bb_{s}\rvert\right) \leq \frac{8}{\bb^{2}}\left(-\bar{\alpha}\varepsilon^{2} + 2\alpha^{2}\right) < 0,
\end{align*}
\noindent for $\bar{\alpha}$ sufficiently large, with $\varepsilon$ satisfying Lemma \ref{consistencyRFassumption1}. The same argument shows the existence of a uniform lower bound.
\end{proof}
We may now improve Lemma \ref{bssquareminus4}.
\begin{lemma}\label{bssquaredminus1}
There exists $\alpha> 0$ such that
\[
\frac{1}{\bb}\left(\bb_{s}^{2} - 1\right) \leq \alpha.
\]
\end{lemma}
\begin{proof}
Let us denote $(\bb_{s}^{2} - 1)/b$ by $\varphi$. A direct computation gives
\[
\varphi_{t} = \Delta\varphi - \frac{2\bb_{ss}^{2}}{\bb} - 3\frac{\bb_{s}^{4}}{\bb^{3}} + \frac{\bb_{s}^{2}}{\bb^{3}}\left(13 - 14\frac{\cc^{2}}{\bb^{2}} \right) - 2\frac{\bb_{s}^{2}\cc_{s}^{2}}{\bb\cc^{2}} + 8\frac{\cc\bb_{s}\cc_{s}}{\bb^{4}} - \frac{4}{\bb^{3}} + 2\frac{\cc^{2}}{\bb^{5}}.
\]
\noindent Since $\varphi$ is uniformly bounded from above at the origin and at spatial infinity we take $(p_{0},t_{0})$ to be the maximum space-time point where $\varphi = \bar{\alpha}$ for the first time, for some positive $\bar{\alpha}$ to be chosen below. We have 
\[
\varphi_{t}(p_{0},t_{0}) \leq \frac{1}{\bb^{2}}\left \{-2\frac{\bb\bb_{s}^{2}\cc_{s}^{2}}{\cc^{2}} + 8\frac{\cc\bb_{s}\cc_{s}}{\bb^{2}} - \frac{7\bb_{s}^{4}}{2\bb} + 14\frac{\bb_{s}^{2}}{\bb}\left(1 - \frac{\cc^{2}}{\bb^{2}}\right)  - \frac{9}{2\bb} + 2\frac{\cc^{2}}{\bb^{3}} \right \}.
\]
\noindent According to Lemma \ref{rotationalsymmetryorder1} we can bound $\cc_{s}$ in terms of $\bb_{s}$. It follows that there exists some positive constant $\alpha > 0$ such that
\begin{align*}
\varphi_{t}(p_{0},t_{0}) &\leq \frac{1}{\bb^{2}}\left \{\alpha - \frac{11\bb_{s}^{4}}{2\bb} + \frac{\bb_{s}^{2}}{\bb}\left(14 - 6\frac{\cc^{2}}{\bb^{2}}\right)  - \frac{9}{2\bb} + 2\frac{\cc^{2}}{\bb^{3}} \right \} \\ &\leq \frac{1}{2\bb^{2}}\left \{\alpha - \frac{11\bb_{s}^{2}}{\bb}\left(\bb_{s}^{2} - 1 \right) + \frac{5}{\bb}\left(\bb_{s}^{2} - 1 \right) + \frac{12\bb_{s}^{2}}{\bb}\left(1 - \frac{\cc^{2}}{\bb^{2}}\right)  - \frac{4}{\bb}\left(1 - \frac{\cc^{2}}{\bb^{2}} \right)\right \}.
\end{align*}
\noindent We finally use Lemma \ref{firstderivativesbounded}, Lemma \ref{rotationalsymmetry} and the fact that $\varphi > 0$ implies $\bb_{s}^{2} > 1$ to derive
\[
\varphi_{t}(p_{0},t_{0}) \leq \frac{1}{2\bb^{2}}\left (\alpha - 6\bar{\alpha}\right) < 0,
\]
\noindent for $\bar{\alpha}$ large enough.
\end{proof}
Next, we extend the previous arguments to the second spatial derivatives. We show that away from the origin the mixed sectional curvatures are controlled by $\cc$ and hence by $\bb$ as in Corollary \ref{verticalsectionalcontrolled}. The flow is singular at some radius $x$ if and only if both $\cc(x,t)$ and $\bb(x,t)$ converge to zero as $t$ approaches $T$.  
\\ The evolution equations of the mixed sectional curvatures \eqref{sectionalhorizontal01} and \eqref{sectionalhorizontal03} are given by
\begin{align}\label{evolutionk01}
\begin{split}
(k_{01})_{t} &= \Delta (k_{01}) + 2k_{01}^{2} + k_{01}\left (\frac{8}{\bb^{2}} -\frac{8\cc^{2}}{\bb^{4}} - \frac{2\cc_{s}^{2}}{\cc^{2}} -\frac{4\bb_{s}^{2}}{\bb^{2}} \right ) + k_{03}\left (\frac{4\cc^{2}}{\bb^{4}} - \frac{2\bb_{s}\cc_{s}}{\bb\cc}\right) \\  &-\frac{4\cc_{s}^{2}}{\bb^{4}} + \frac{24\cc\bb_{s}\cc_{s}}{\bb^{5}} -\frac{2\bb_{s}\cc_{s}^{3}}{\bb\cc^{3}} -\frac{24\cc^{2}\bb_{s}^{2}}{\bb^{6}} + \frac{8\bb_{s}^{2}}{\bb^{4}} -\frac{2\bb_{s}^{4}}{\bb^{4}}
\end{split}
\end{align}
\noindent and
\begin{equation}\label{evolutionk03}
\begin{split}
(k_{03})_{t} &= \Delta (k_{03}) + 2k_{03}^{2} - 4k_{03}\left (\frac{\bb_{s}^{2}}{\bb^{2}} +\frac{\cc^{2}}{\bb^{4}}\right ) + 4k_{01}\left (\frac{2\cc^{2}}{\bb^{4}} - \frac{\bb_{s}\cc_{s}}{\bb\cc}\right) \\  & + \frac{12\cc_{s}^{2}}{\bb^{4}} +\frac{40\cc^{2}\bb_{s}^{2}}{\bb^{6}} -\frac{48\cc\bb_{s}\cc_{s}}{\bb^{5}} - \frac{4\bb_{s}^{3}\cc_{s}}{\bb^{3}\cc}.
\end{split}
\end{equation}
\\
\begin{lemma}\label{2ndderivativescontrolleb2}
There exists $\alpha > 0$ such that
\[
\lvert k_{01} \rvert + \lvert k_{03} \rvert \leq \frac{\alpha}{\bb^{2}}.
\]
\end{lemma}
\begin{proof}
By \cite{shi} there exists $\alpha > 0$ such that $\lvert k_{01}\rvert \leq \alpha$ in $\R^{4}\times [0,T/2]$. Since by Lemma \ref{finite-timesingularity} $\bb$ is uniformly bounded from above, we deduce that $\bb^{2}\lvert k_{01}\rvert \leq \alpha$ and similarly for $\bb^{2}\lvert k_{03}\rvert$ using Lemma \ref{consistencyRFassumption1}. We may hence consider the case $t\in [T/2,T)$. Define the map $\psi \doteq \bb\bb_{ss} - \mu \bb_{s}^{2} - \nu\cc_{s}^{2}$ in $\R^{4}\times [T/2,T)$, for some $\mu$ and $\nu$ positive constants we will determine below. According to Claim \ref{secondderivativesdecay} $\psi(x,t)\rightarrow 0$ as $x\rightarrow \infty$ for any $t\in [T/2,T)$. We then adapt the argument in \cite[Lemma 7]{IKS2} to show that $\psi$ is uniformly bounded in the space-time region. Explicitly, at any stationary point of $\psi(\cdot,t)$ we have
\begin{align*}
\psi_{t} &= \Delta\psi -\bb\bb_{ss}\left(\frac{4\cc^{2}}{\bb^{4}} + \frac{2\cc_{s}^{2}}{\cc^{2}} + \frac{4\mu\bb_{s}^{2}}{\bb^{2}} \right) - \frac{16\nu\cc^{3}\bb_{s}\cc_{s}}{\bb^{5}}  - (24 + 8\mu)\frac{\cc\bb_{s}\cc_{s}}{\bb^{3}}  + \frac{2\bb\bb_{s}\cc_{s}^{3}}{\cc^{3}}  \\ &- \frac{8\bb_{s}^{2}(\mu + 1)}{\bb^{2}}+ \frac{4\cc\cc_{ss}}{\bb^{2}} + \frac{4\nu\cc_{s}^{2}\cc_{ss}}{\cc} - \frac{8\nu\bb_{s}\cc_{s}\cc_{ss}}{\bb} - \frac{2\bb\bb_{s}\cc_{s}\cc_{ss}}{\cc^{2}} + \frac{12\nu\cc^{2}\cc_{s}^{2}}{\bb^{4}} + \frac{4\cc_{s}^{2}}{\bb^{2}} \\ &+ \frac{24\cc^{2}\bb_{s}^{2}}{\bb^{4}} + \frac{12\mu\cc^{2}\bb_{s}^{2}}{\bb^{4}} + \frac{2\mu\bb_{s}^{2}\cc_{s}^{2}}{\cc^{2}} + \frac{4\nu\bb_{s}^{2}\cc_{s}^{2}}{\bb^{2}} + \frac{2\bb_{s}^{4}}{\bb^{2}} + 2\nu\cc_{ss}^{2} + 2(\mu -1)\bb_{ss}^{2}.  
\end{align*}
\noindent Suppose $\psi$ attains some negative value $-\bar{\alpha}$ for the first time at $t_{0}\in [T/2,T)$. By Lemma \ref{firstderivativesbounded} we can choose $\bar{\alpha}$ sufficiently large such that  $\bb\bb_{ss} \leq -\bar{\alpha}/4$. Therefore we get 
\[
-\bb\bb_{ss} \left(\frac{4\cc^{2}}{\bb^{4}} + \frac{2\cc_{s}^{2}}{\cc^{2}} + \frac{4\mu\bb_{s}^{2}}{\bb^{2}} \right) \geq \bar{\alpha}\left(\frac{\cc^{2}}{2\bb^{4}} + \frac{\cc_{s}^{2}}{2\cc^{2}} + \frac{\mu\bb_{s}^{2}}{\bb^{2}}\right).
\]
\noindent Evaluating the evolution equation of $\psi$ at $(p_{0},t_{0})$, we have (provided we set $\mu \geq 1$)
\begin{align*}
\psi_{t}(p_{0},t_{0}) &\geq 2\nu\cc_{ss}^{2} + \bar{\alpha}\left(\frac{\cc^{2}}{2\bb^{4}} + \frac{\cc_{s}^{2}}{2\cc^{2}} + \frac{\mu\bb_{s}^{2}}{\bb^{2}}\right) - \frac{16\nu\cc^{3}\bb_{s}\cc_{s}}{\bb^{5}}  - (24 + 8\mu)\frac{\cc\bb_{s}\cc_{s}}{\bb^{3}}   \\ &+ \frac{2\bb\bb_{s}\cc_{s}^{3}}{\cc^{3}} - \frac{8\bb_{s}^{2}(\mu + 1)}{\bb^{2}}+ \frac{4\cc\cc_{ss}}{\bb^{2}} + \frac{4\nu\cc_{s}^{2}\cc_{ss}}{\cc} - \frac{8\nu\bb_{s}\cc_{s}\cc_{ss}}{\bb} - \frac{2\bb\bb_{s}\cc_{s}\cc_{ss}}{\cc^{2}}.
\end{align*}
\noindent One can then estimate the remaining terms exactly as in \cite{IKS2} by using the uniform boundedness of the first spatial derivatives and Lemma \ref{consistencyRFassumption1}. Namely, we can use the weighted Cauchy-Schwarz inequality to get
\[
\left \vert \frac{2\bb\bb_{s}\cc_{s}\cc_{ss}}{\cc^{2}} \right \vert \leq \frac{\cc_{ss}^{2}}{2} + 2\frac{\bb^{2}\bb_{s}^{2}\cc_{s}^{2}}{\cc^{4}} \leq \frac{\cc_{ss}^{2}}{2} + 2\frac{\alpha^{2}\bb^{2}\cc_{s}^{2}}{\cc^{4}} \leq \frac{\cc_{ss}^{2}}{2} + 2\frac{\alpha^{2}\varepsilon^{-2}\cc_{s}^{2}}{\cc^{2}},
\]
\noindent and similarly for the others. Therefore there exists a uniform constant $\alpha$ such that the time derivative at $(p_{0},t_{0})$ is bounded from below as
\[
\psi_{t}(p_{0},t_{0}) \geq (\nu - 1)\cc_{ss}^{2} + \bar{\alpha}\left(\frac{\cc^{2}}{2\bb^{4}} + \frac{\cc_{s}^{2}}{2\cc^{2}} + \frac{\mu\bb_{s}^{2}}{\bb^{2}}\right) - \alpha(\mu + \nu + 1)\left(\frac{\cc^{2}}{2\bb^{4}} + \frac{\cc_{s}^{2}}{2\cc^{2}} + \frac{\mu\bb_{s}^{2}}{\bb^{2}} \right) > 0,
\]
\noindent once we choose $\mu = 1, \nu = 2$ and $\bar{\alpha}$ sufficiently large. The existence of a uniform upper bound follows from a similar argument by considering $\tilde{\psi} \doteq \bb\bb_{ss} + \mu\bb_{s}^{2} + \nu\bb_{s}^{2}$. The very same proof applies for $k_{03}$. 
\end{proof}
We can finally show that both $\cc$ and $\bb$ admit limits as the flow approaches its maximal time of existence $T < \infty$.
\begin{corollaryy}\label{cadmitslimit}
For any $x\geq 0$ the limits $\lim_{t\nearrow T}\cc(x,t)$ and $\lim_{t\nearrow T}\bb(x,t)$ exist and are finite.
\end{corollaryy}
\begin{proof}
By applying Lemmas \ref{firstderivativesbounded} and \ref{2ndderivativescontrolleb2} we get
\[
\lvert(\cc^{2})_{t}\rvert \leq  2\lvert\cc\cc_{ss}\rvert + 4 \left \vert\frac{\cc}{\bb}\bb_{s}\cc_{s} - \frac{\cc^{4}}{\bb^{4}}\right \vert \leq \alpha.
\]
\noindent The same argument works for $\bb$ as well.
\end{proof}
The curvature is hence uniformly controlled along any Euclidean hypersphere where the components $\bb$ and $\cc$ do not converge to zero as $t\nearrow T$. Namely, from Corollary \ref{verticalsectionalcontrolled} and Lemma \ref{2ndderivativescontrolleb2} it follows that there exists a positive constant $\alpha$ such that  
\begin{equation}\label{normsquaredcurvatureboundedbyc}
\sup_{\R^{4}\times [0,T)}\bb^{2}\lvert \text{Rm}_{g(t)}\rvert_{g(t)} \leq \alpha.
\end{equation}
\noindent Next, we show that around a singularity rotational symmetry type of bounds hold for the second spatial derivatives as well. 
\begin{lemma}\label{rotationalsymmetry2}
There exists $\alpha > 0$ such that
\[
\lvert k_{01} - k_{03} \rvert \leq \frac{\alpha}{\bb}.
\]
\end{lemma}
\begin{proof}
We adapt the proof in \cite{IKS1}, whose argument works for a compact Type-I Ricci flow setting. Once we define $\psi\doteq \cc_{s}/\cc - \bb_{s}/\bb$, we consider the map 
\[
\varphi \doteq (\psi_{s}\bb)^{2} \equiv \left (\frac{\cc_{ss}}{\cc} - \frac{\cc_{s}^{2}}{\cc^{2}} - \frac{\bb_{ss}}{\bb} + \frac{\bb_{s}^{2}}{\bb^{2}} \right)^{2}\bb^{2}.
\]
\noindent The boundary conditions \eqref{smoothnessorigin} ensure that $\varphi(\origin,t) = 0$ for any $t$. From the proof of Lemma \ref{firstderivativesbounded} and Claim \ref{secondderivativesdecay} we derive that if $\varphi$ is not bounded, then for any sufficiently large value $\bar{\alpha}$ there exist $t_{0}\geq T/2$ and a maximum point $p_{0}$ such that $\varphi(p_{0},t_{0})= \bar{\alpha}$ for the first time. The evolution equation for $\varphi$ is 
\[
\varphi_{t} = \Delta \varphi - 2(\bb^{2})_{s}(\psi_{s}^{2})_{s} -2\bb^{2}\psi_{ss}^{2} - 2F\varphi - 2G\bb^{2}\psi\psi_{s} + 2 H\bb^{2}\psi_{s} - 4\bb_{s}^{2}\psi_{s}^{2} - 8\psi_{s}^{2} + 4\psi_{s}^{2}\frac{\cc^{2}}{\bb^{2}}, 
\]
\noindent where 
\[
F\doteq 4\frac{\bb_{s}^{2}}{\bb^{2}} + 2 \frac{\cc_{s}^{2}}{\cc^{2}} + 8 \frac{\cc^{2}}{\bb^{4}}, \,\,\,\,\,\,\, G \doteq \left(2\frac{\bb_{s}^{2}}{\bb^{2}} + \frac{\cc_{s}^{2}}{\cc^{2}} + 8 \frac{\cc^{2}}{\bb^{4}} \right)_{s} \,\,\,\,\,\,\, H \doteq -\left(8\frac{\bb_{s}}{\bb^{3}}\left(1 - \frac{\cc^{2}}{\bb^{2}}\right) \right)_{s}.
\]
\noindent Evaluating $\varphi$ at the maximum point $(p_{0},t_{0})$ we get
\begin{align*}
\varphi_{t}(p_{0},t_{0}) &\leq 2\bb_{s}^{2}\psi_{s}^{2} - \psi_{s}^{2}\left(8 - 4 \frac{\cc^{2}}{\bb^{2}}\right) + - 2F\varphi - 2G\bb^{2}\psi\psi_{s} + 2 H\bb^{2}\psi_{s} \\ &\leq - \psi_{s}^{2}\left(8 + 12 \frac{\cc^{2}}{\bb^{2}}\right) - 2G\bb^{2}\psi\psi_{s} + 2 H\bb^{2}\psi_{s}.
\end{align*}
\noindent From \eqref{normsquaredcurvatureboundedbyc} it easily follows that there exists some uniform constant $\alpha > 0$ such that $\lvert G \rvert \leq \alpha/\bb^{3}$. Being $\psi$ uniformly bounded (Lemma \ref{rotationalsymmetryorder1}), we have
\[
- 2G\bb^{2}\psi\psi_{s} \leq \alpha \frac{\lvert\psi_{s}\rvert}{\bb}.
\]
\noindent According to Lemma \ref{rotationalsymmetry} and Lemma \ref{rotationalsymmetryorder1} an analogous estimate can be found for $\lvert H \bb^{2}\psi_{s}\rvert$. Then 
\[
\varphi_{t}(p_{0},t_{0})\leq - \psi_{s}^{2}\left(8 + 12 \frac{\cc^{2}}{\bb^{2}}\right) + \alpha \frac{\lvert\psi_{s}\rvert}{\bb} \equiv \frac{\lvert \psi_{s}\rvert}{b}\left(\left(-8 -12\frac{\cc^{2}}{\bb^{2}}\right)\sqrt{\bar{\alpha}} + \alpha \right) < 0, 
\]
\noindent for $\bar{\alpha}$ sufficiently large. Therefore $\varphi$ is uniformly bounded and we get
\begin{align*}
\bb\lvert k_{01} - k_{03} \rvert &\leq \alpha + \left \vert \frac{\bb_{s}^{2}}{\bb} - \frac{\bb}{\cc}\frac{\cc_{s}^{2}}{\cc}\right \vert \\ &\leq \alpha + \lvert \bb_{s}\rvert \left \vert \frac{\bb_{s}}{\bb} - \frac{\cc_{s}}{\cc} \right \vert  + \lvert \cc_{s}\rvert\frac{\bb}{\cc} \left \vert \frac{\bb_{s}}{\bb} - \frac{\cc_{s}}{\cc} \right \vert \leq \alpha\left(1 + \alpha + \lvert \cc_{s}\rvert \frac{\bb}{\cc}\right) \leq \alpha\left( 1 + \varepsilon^{-1})\right),
\end{align*}
\noindent where we have used Lemma \ref{rotationalsymmetryorder1} and Lemma \ref{consistencyRFassumption1}. 
\end{proof}
We finally discuss the existence of lower bounds for the mixed sectional curvatures.
\begin{lemma}\label{estimatelog}
There exists $\alpha>0$ such that
\[
\cc_{ss}\cc\,\emph{log}\,c \geq - \alpha, \,\,\,\,\,\,\, \bb_{ss}\bb\,\emph{log}\,b \geq - \alpha.
\]
\end{lemma}
\begin{proof}
We adapt the analogous argument in \cite{IKS1}. Consider the map $f \doteq \cc_{ss}\cc\,\text{log}\,c$, which is smooth in $\mathbb{R}^{4}\setminus \{\origin\}\times [0,T)$. Moreover $f$ extends continuously to the origin and $f(\origin,t) = 0$ as long as a solution exists. By Claim \ref{secondderivativesdecay} and the fact that $\cc$ is uniformly bounded from above, we deduce that $f(x,t)\rightarrow 0$ at spatial infinity for any $t\in[T/2,T)$. Suppose that there exist $(p_{0},t_{0})\in(\R^{4}\setminus\{\origin\})\times [T/2,T)$ and $\bar{\alpha}$ large to be chosen below such that $f(p_{0},t_{0}) = - \bar{\alpha}$ for the first time.
From \eqref{sectionalhorizontal03} and \eqref{normsquaredcurvatureboundedbyc} it follows 
\[
\bar{\alpha} = \lvert f(p_{0},t_{0}) \rvert \leq \alpha \lvert \text{log}\,\cc(p_{0},t_{0})\rvert.
\]
\noindent Since by Lemma \ref{finite-timesingularity} $\cc$ is uniformly bounded from above the last inequality implies $\text{log}\,\cc(p_{0},t_{0}) < 0$ and
\begin{equation}\label{controlcsslogestimate}
\cc_{ss}(p_{0},t_{0}) = \frac{\bar{\alpha}}{\cc\lvert\text{log}\,\cc\rvert}(p_{0},t_{0}) \geq \bar{\alpha}.
\end{equation}
\noindent for $\bar{\alpha}$ large enough. By direct computation we get
\begin{align*}
f_{t} &= \Delta f - 2\left(2 + \frac{1}{\text{log}\,\cc} \right)\frac{\cc_{s}}{\cc}f_{s} -\cc\,\text{log}\,\cc\left(12\frac{\cc\cc_{s}^{2}}{\bb^{4}} - 48\frac{\cc^{2}}{\bb^{5}}\bb_{s}\cc_{s} + 40 \frac{\cc^{3}}{\bb^{6}}\bb_{s}^{2} - 4\frac{\bb_{s}^{3}\cc_{s}}{\bb^{3}} \right) \\ & -8\frac{\cc^{4}}{\bb^{4}}\text{log}\,\cc\left(\frac{\cc_{ss}}{\cc} - \frac{\bb_{ss}}{\bb} \right) - 2\frac{\cc_{ss}}{\cc}\left(\frac{\cc^{4}}{\bb^{4}} + f \right) + 2\frac{\cc_{s}^{2}\cc_{ss}}{\cc}\left(2 + \frac{1}{\text{log}\,\cc}\right) \\ &-4\cc\,\text{log}\,\cc\left( \frac{\cc_{ss}\bb_{s}^{2}}{\bb^{2}} + \frac{\bb_{ss}\bb_{s}\cc_{s}}{\bb^{2}} - \frac{\cc_{ss}\cc_{s}^{2}}{\cc^{2}}\right).
\end{align*}
\noindent In the following the signs of $\bb_{s}$ and $\cc_{s}$ are not relevant. In particular we assume without loss of generality that $\cc_{s}\geq 0$. By Lemma \ref{rotationalsymmetryorder1} we have 
\[
12\frac{\cc\cc_{s}^{2}}{\bb^{4}} - 48\frac{\cc^{2}}{\bb^{5}}\bb_{s}\cc_{s} + 40 \frac{\cc^{3}}{\bb^{6}}\bb_{s}^{2} - 4\frac{\bb_{s}^{3}\cc_{s}}{\bb^{3}} \geq \frac{4}{\bb^{2}}\left(-\alpha - \bb_{s}^{2}\frac{\cc}{\bb^{2}}\left(\bb_{s}^{2} - \frac{\cc^{2}}{\bb^{2}} \right) \right) \geq -\frac{\alpha}{\cc^{2}},
\]
\noindent where we have used Lemma \ref{rotationalsymmetry} and Lemma \ref{bssquaredminus1} to derive the last inequality. According to Lemma \ref{rotationalsymmetry2} it holds
\[
-8\frac{\cc^{4}}{\bb^{4}}\text{log}\,\cc\left(\frac{\cc_{ss}}{\cc} - \frac{\bb_{ss}}{\bb} \right) \geq -\alpha\frac{\lvert\text{log}\,\cc\rvert}{\cc}.
\]
\noindent By choosing $\bar{\alpha}$ large enough (and hence $\cc(p_{0},t_{0})$ small) it follows that $2(\cc_{s}\cc_{ss}/\cc)(2+1/\text{log}\,\cc)(p_{0},t_{0}) \geq 0$. Finally, Lemma \ref{rotationalsymmetry}, Lemma \ref{rotationalsymmetryorder1}, and Lemma \ref{rotationalsymmetry2} yield the bounds
\begin{align*}
&\left(\frac{\cc_{ss}\bb_{s}^{2}}{\bb^{2}} + \frac{\bb_{ss}\bb_{s}\cc_{s}}{\bb^{2}} - \frac{\cc_{ss}\cc_{s}^{2}}{\cc^{2}}\right)(p_{0},t_{0}) \geq \left(\frac{\bb_{ss}\bb_{s}\cc_{s}}{\bb^{2}} - \frac{\cc_{ss}\cc_{s}^{2}}{\cc^{2}}\right)(p_{0},t_{0}) = \\  &\left(\frac{\bb_{s}\cc_{s}}{\bb}\left(\frac{\bb_{ss}}{\bb} - \frac{\cc_{ss}}{\cc} \right) + \frac{\cc_{ss}}{\cc}\frac{\cc_{s}(\cc(\bb_{s}-\cc_{s}) + \cc_{s}(\cc-\bb))}{\bb\cc}\right)(p_{0},t_{0}) \geq \left(-\frac{\alpha}{\cc^{2}} - \alpha\frac{\cc_{ss}}{\cc}\right)(p_{0},t_{0}).
\end{align*}
\noindent Evaluating the evolution equation of $f$ at $(p_{0},t_{0})$ and using \eqref{controlcsslogestimate} we get the lower bound
\begin{align*}
f_{t}(p_{0},t_{0}) &\geq -\alpha\left(\frac{\lvert \text{log}\,\cc\rvert}{\cc} + \lvert \text{log}\,\cc\rvert \cc_{ss}\right) + 2\frac{\cc_{ss}}{\cc}\left( \bar{\alpha} - 1 \right) = \frac{1}{\cc}\left(-\alpha(\lvert \text{log}\,\cc\rvert + \bar{\alpha}) + 2\cc_{ss}(\bar{\alpha} - 1)\right) \\ & = \frac{1}{\cc}\left(-\alpha\lvert \text{log}\,\cc\rvert + \bar{\alpha}(\cc_{ss} - \alpha) + \cc_{ss}\frac{\bar{\alpha}}{2} + \frac{\cc_{ss}}{2}(\bar{\alpha} - 4)\right) \geq \frac{1}{\cc}\left(-\alpha\lvert \text{log}\,\cc\rvert + \cc_{ss}\frac{\bar{\alpha}}{2}\right) \\ & = \frac{1}{2\lvert \text{log}\,\cc\rvert\cc^{2}}\left(-2\alpha\lvert \text{log}\,\cc\rvert^{2}\cc + \bar{\alpha}^{2}\right) > 0,
\end{align*}
\noindent for $\bar{\alpha}$ sufficiently large (i.e. $\cc$ small enough). The case of $\bb_{ss}\bb\,\text{log}\,\bb$ does not require modifications. 
\end{proof}
\subsection{Curvature estimates in $\Gin$.} We consider $(\R^{4},g(t))_{0\leq t < T}$ a maximal complete, bounded curvature Ricci flow solution evolving from some $g_{0}\in\Gin$. If the solution develops a finite-time singularity at some $T < \infty$, then we can apply \cite[Theorem 1.1]{pseudolocalityapplication} and conclude that there exists $\rho > 0$ such that 
\begin{equation}\label{definitionradiusrho}
\sup_{\left(\R^{4}\setminus B(\origin,\rho)\right)\times [0,T)}\left \vert \text{Rm}_{g(t)}\right \vert_{g(t)}\leq 1.
\end{equation}
\begin{remark}
We note that the set $\Gin$ is preserved along the Ricci flow. Consider the maximal Ricci flow evolving from some $g_{0}\in\Gin$. By Lemma \ref{cruciallemma} condition (i) in Definition \ref{definitionGin} persists in time. From \cite{formationsingularities} we also derive that $\lvert\text{Rm}\rvert(s,t)\rightarrow 0$ as $s\rightarrow \infty$ for all $t\in [0,T)$. If $T<\infty$, then given $\rho$ as in \eqref{definitionradiusrho} we can find $\mu > 0$ such that $\cc(\cdot,t) \geq \mu > 0$ in $\R^{4}\setminus B(\origin,\rho)\times [0,T)$. If instead $T=\infty$, then $\cc$ is uniformly bounded from below away from the origin in any compact interval of existence being the curvature bounded. 
\end{remark}
\noindent An immediate consequence of \eqref{definitionradiusrho} is that for any radial coordinate $x_{1} > \rho$ the spatial derivatives (up to second order) of $\bb$ and $\cc$ are uniformly bounded in time along the hypersphere of radius $x_{1}$. In particular, for any $x_{1} > \rho$ there exists $\varepsilon = \varepsilon(x_{1}) > 0$ satisfying
\begin{equation}\label{gammadefinition}
\inf_{p\in B(\origin,x_{1})}\frac{\cc}{\bb}(p,t) \geq \varepsilon > 0,
\end{equation}
\noindent for any $t\in [0,T)$. For the function $\cc/\bb$ is uniformly bounded from below at the origin and along the hypersphere of radius $x_{1}$ and the evolution equation \eqref{evolutionlogcb} prevents $\cc/\bb$ from attaining interior minima approaching zero.   
\\By inspection one can check that given $x_{1} > \rho$ any bound derived in Subsection 4.1 extends to the space-time region $B(\origin,x_{1})\times [0,T)$. For any argument relies on a maximum principle which still applies to this setting once we know that any relevant quantity is uniformly bounded along the parabolic boundary of the region $B(\origin,x_{1})\times [0,T)$. Explicitly, we have the following
\begin{lemma}\label{estimatesforGin}
Let $(\R^{4},g(t))_{0\leq t < T}$, with $T < \infty$, be the maximal Ricci flow solution evolving from some $g_{0}\in\Gin$ and let $\rho > 0$ satisfy
\[
\sup_{\left(\R^{4}\setminus B(\origin,\rho)\right)\times [0,T)}\left \vert \emph{Rm}_{g(t)}\right \vert_{g(t)}\leq 1.
\]
\noindent Then for any $x_{1} > \rho$ there exists $\alpha = \alpha(x_{1}) > 0$ such that 
\[
\sup_{B(\origin,x_{1})\times [0,T)}\bb^{2}\lvert \emph{Rm}_{g(t)}\rvert_{g(t)}\leq \alpha, 
\]
\noindent and
\[
\sup_{B(\origin,x_{1})\times [0,T)}\left(\frac{1}{\bb}\left(\frac{\bb}{\cc} -1 \right) + \left \vert\frac{\cc_{s}}{\cc} -\frac{\bb_{s}}{\bb} \right\vert + \bb\lvert k_{01} - k_{03} \rvert\right) \leq \alpha.
\]
\end{lemma}
\begin{remark}
One can verify that Lemma \ref{estimatesforGin} holds for a larger class of Ricci flows than $\Gin$. Indeed, it suffices to control the flow uniformly along the parabolic boundary of some space-time region and then apply maximum principle arguments without relying on the quantities $\bb_{s}$ and $H$ being nonnegative.
\end{remark}
Next, we prove that $\bb_{s}$ and $H$ remain positive along a hypersphere of sufficiently large radius. In the following $\rho$ still denotes the radius defined by \eqref{definitionradiusrho}. 
\begin{lemma}\label{definitionxtilde}
Let $(\R^{4},g(t))_{0\leq t < T}$, with $T < \infty$, be the maximal Ricci flow solution evolving from a warped Berger metric $g_{0}\in\Gin$. There exist $\tilde{x}_{2}\geq \tilde{x_{1}} > \rho$, $\delta > 0$ and $\tilde{t}\in [0,T)$ such that
\[
\bb_{s}(\tilde{x}_{1},t)\geq \delta > 0, \,\,\,\,\,\,\,\, H(\tilde{x}_{2},t) \geq \delta > 0 \,\,\,\,\,\,\,\, \forall t\in [\tilde{t},T).
\] 
\end{lemma}
\begin{proof}
We have already shown that $\bb(x,0)\rightarrow \infty$ as $x\rightarrow\infty$. Since the curvature is uniformly bounded in the complement of the Euclidean ball $B(\origin,\rho)$, we can pick $\rho < x_{0} < x_{1}$ such that
\[
\bb(x_{1},t) - \bb(x_{0},t) \geq \varsigma > 0,
\]
\noindent for some $\varsigma > 0$ and for any $t\in [0,T)$. We can use the Koszul formula to write the evolution equation of $\bb_{s}$ as 
\begin{align}
\partial_{t}(\bb_{s})&= \partial_{s}\left(-\text{Ric}_{g(t)}\left(\frac{X_{1}}{\bb},\frac{X_{1}}{\bb}\right)\bb\right) + \text{Ric}_{g(t)}(\partial_{s},\partial_{s})\bb_{s} \notag \\ &= -\text{Ric}_{g(t)}\left(\frac{X_{1}}{\bb},\frac{X_{1}}{\bb}\right)\bb_{s} -\nabla_{g(t)}\text{Ric}_{g(t)}\left(\partial_{s},\frac{X_{1}}{\bb},\frac{X_{1}}{\bb}\right)\bb + \text{Ric}_{g(t)}(\partial_{s},\partial_{s})\bb_{s}. \label{firstderivativelipschitz}
\end{align}
\noindent Therefore given $x > \rho$, by \eqref{definitionradiusrho} and Shi's derivative estimates there exists $\alpha(x) > 0$ such that $\lvert(\bb_{s})_{t}(x,t)\rvert \leq \alpha$ uniformly in time. Since $T < \infty$ the last property implies that $\bb_{s}(x,\cdot)$ is Lipschitz and hence admits a finite limit as $t\nearrow T$, which we know to be nonnegative according to Lemma \ref{cruciallemma}. Let us assume for a contradiction that any such limit is zero. Since $\bb_{ss}$ is bounded in the annular region $(x_{0},x_{1})\times S^{3}$ uniformly in time, we deduce that $\sup_{[x_{0},x_{1}]}\bb_{s}(\cdot,t)\rightarrow 0$ as $t\nearrow T$. On the other hand, being the curvature controlled in the annular region $(x_{0},x_{1})\times S^{3}$, we get
\begin{align*}
\varsigma &\leq \bb(x_{1},t) - \bb(x_{0},t) \leq \sup_{[x_{0},x_{1}]}\bb_{s}(\cdot,t)(s(x_{1},t)-s(x_{0},t))\\ &\leq \alpha\,\sup_{[x_{0},x_{1}]}\bb_{s}(\cdot,t)(s(x_{1},0)-s(x_{0},0))\leq \alpha\,\sup_{[x_{0},x_{1}]}\bb_{s}(\cdot,t),
\end{align*}
\noindent for any $t\in[0,T)$, which gives a contradiction. Therefore, there exists $\tilde{x_{1}}\in[x_{0},x_{1}]$ as in the statement. The proof for $H$ is similar. Indeed we can write $H = (\log(\bb^{2}\cc))_{s}$ and then adapt the argument above noting that by Definition \ref{definitionGin} $\log(\bb^{2}\cc))(x,0)\rightarrow \infty$ as $x\rightarrow \infty$. In particular we can always pick $\tilde{x}_{2}\geq \tilde{x}_{1}$. 
\end{proof}
Next we show that $\cc H$ stays away from zero in the compact region $B(\origin,\rho)$ for times close to the maximal time of existence $T$. In the following we let $\tilde{x}_{2}$ and $\tilde{t}$ be defined as in the previous Lemma. 
\begin{corollaryy}\label{lowerboundH}
There exists $\mu > 0$ such that $\cc H \geq \mu$ in $B(\origin,\tilde{x}_{2})\times [\tilde{t},T)$. 
\end{corollaryy}
\begin{proof}
By the boundary conditions we have $\cc H(\origin,t) = 3$ as long as the solution exists. According to Lemma \ref{cruciallemma} and Lemma \ref{definitionxtilde} there exists $\delta > 0$ such that $\cc H(\tilde{x}_{2},t)\geq \delta$ for any $t\in [\tilde{t},T)$ and $\cc H(x,\tilde{t})\geq \delta$ for any $0\leq x\leq \tilde{x}_{2}$. Suppose that $cH$ gets smaller than $\min\{\delta,3\}$ in $B(\origin,\tilde{x}_{2})\times [\tilde{t},T)$. Then there exists a minimum point $(p_{0},t_{0})$ and from \eqref{evolutioncH}, \eqref{gammadefinition} and Lemma \ref{cruciallemma} we get
\[
(\cc H)_{t}(p_{0},t_{0}) \geq 2\frac{\cc H}{\bb^{2}}\left( \frac{\cc^{2}}{\bb^{2}} - \bb_{s}^{2}\right)(p_{0},t_{0}) \geq 2\frac{\cc H}{\bb^{2}}\left( \varepsilon^{2} - \bb_{s}^{2}\right)(p_{0},t_{0}).
\]
\noindent From Lemma \ref{rotationalsymmetryorder1} it follows that
\[
\bb_{s} = \frac{1}{2}\left(\frac{\bb}{\cc}\cc H - \frac{\bb}{\cc}\cc_{s}\right) \leq \frac{1}{2}\left(\frac{\bb}{\cc}\cc H - \bb_{s} + \alpha\bb\right),
\]
\noindent for some $\alpha = \alpha(\tilde{x}_{2}) > 0$. Thus we can find a uniform constant $\alpha$ only depending on $\tilde{x}_{2}$ such that
\[
(\cc H)_{t}(p_{0},t_{0}) \geq \frac{2\cc H}{\bb^{2}}\left(\varepsilon^{2} - \alpha\cc H(1 + \cc H) -\alpha\bb^{2}\right)(p_{0},t_{0}). 
\]
\noindent Therefore, whenever $\cc H\leq \tilde{\delta}$, for some $\tilde{\delta}$ only depending on $\tilde{x}_{2}$ and $\tilde{t}$, the function $t\mapsto \min_{B(\origin,\tilde{x}_{2})}(\cc H)(\cdot,t)$ is Lipschitz and satisfies
\[
\frac{d(\cc H)_{\text{min}}}{dt} \geq -\alpha (\cc H)_{\text{min}}.
\]
\noindent We conclude that $\cc H$ cannot approach zero in the interior of $B(\origin,\tilde{x}_{2})$ as $t
\nearrow T$.
\end{proof} 
\section{Singularity models of warped Berger Ricci flows}
In this section we perform a blow-up analysis of warped Berger Ricci flows. We first recall the following general notion.
\begin{definitionn}\label{definitonsingularitymodels}
A complete bounded curvature ancient solution to the Ricci flow $(M_{\infty},g_{\infty}(t))_{-\infty < t \leq 0}$ is a \emph{singularity model} for a warped Berger Ricci flow $(\R^{4},g(t))_{0\leq t < T}$ if $ T < \infty$ and there exists a sequence of space-time points $(p_{j},t_{j})$ with $t_{j}\nearrow T$ such that $\lambda_{j}\doteq \lvert\text{Rm}_{g(t_{j})}\rvert_{g(t_{j})}(p_{j})\rightarrow \infty$ and the rescaled Ricci flows $(\R^{4},g_{j}(t),p_{j})$ defined by
\[
g_{j}(t)\doteq \lambda_{j}g\left(t_{j}+\frac{t}{\lambda_{j}}\right)
\] 
\noindent converge to $(M_{\infty},g_{\infty}(t),p_{\infty})$ in the pointed Cheeger-Gromov sense for $t\in(-\infty,0]$. 
\end{definitionn}
\begin{remark}
We note that by the Cheeger-Gromov convergence any singularity model $(M_{\infty},g_{\infty}(t))$ of a warped Berger Ricci flow is \emph{non-compact} and \emph{non-flat}.
\end{remark}
The main goal of this section consists in classifying the singularity models of warped Berger Ricci flows. Namely, we show the following result
\begin{proposition}\label{classificationsingularity}
Let $(\R^{4},g(t))_{0\leq t < T}$, with $T < \infty$, be the maximal Ricci flow solution evolving from a warped Berger metric $g_{0}$ belonging to either $\G$ or $\Gin$. Then any singularity model is either the self-similar shrinking soliton on the cylinder or a positively curved rotationally symmetric $\kappa$-solution. 
\end{proposition}
We prove the characterization of singularity models by showing that the symmetries of warped Berger Ricci flows are enhanced when dilating the flow around a singularity. More precisely, we show that the left-invariant vector fields in \eqref{leftinvariantframe} become Killing vectors when passing to the limit hence forcing the singularity model to be rotationally symmetric.
\\Given $g_{0}\in \G$ there exists $\varepsilon > 0$ such that $\cc/\bb(\cdot,0)\geq \varepsilon$. Therefore $g_{0}$ is bounded between two round cylinders outside some compact region and there exists $\alpha > 0$ such that $\text{Vol}_{g_{0}}(B_{g_{0}}(p,1))\geq \alpha$ for any $p\in\mathbb{R}^{4}$. The latter condition is satisfied by any $g_{0}\in\Gin$ being the injectivity radius positive and the curvature bounded. Thus if $(\R^{4},g(t))_{0\leq t < T}$, with $T < \infty$, is the maximal Ricci flow solution evolving from some $g_{0}$ which belongs to either $\G$ or $\Gin$, then by \cite[Theorem 8.26]{ricciflowtechniques1} there exists $\kappa > 0$ such that $g(t)$ is (weakly) $\kappa-$non-collapsed in $\mathbb{R}^{4}\times (T/2,T)$ at any scale $r\in (0,\sqrt{T/2})$. Accordingly, there exist blow-up sequences satisfying Definition \ref{definitonsingularitymodels} and hence any warped Berger Ricci flow evolving from either $\G$ or $\Gin$ admits singularity models (\cite[Section 16]{formationsingularities}). In particular, any singularity model of a warped Berger Ricci flow is (weakly) $\kappa$-non-collapsed at all scales.
\\We first consider a maximal Ricci flow solution $(\R^{4},g(t))_{0\leq t < T}$, with $T < \infty$, starting at some warped Berger metric $g_{0}\in\G$. Later we check that the same conclusions are satisfied by Ricci flows in $\Gin$.
We let $(p_{j},t_{j})$ be a blow-up sequence of space-time points giving rise to a singularity model $(M_{\infty},g_{\infty}(t),p_{\infty})$ as in Definition \ref{definitonsingularitymodels} and we denote the rescaling factors $\lvert \text{Rm}_{g(t_{j})}\rvert_{g(t_{j})}(p_{j})$ by $\lambda_{j}$. Due to the SU(2)-symmetry we may fix $\bar{\theta}\in S^{3}$ and we may set $p_{j} = (x_{j},\bar{\theta})$.
We also let $(\Phi_{j})$ be the diffeomorphisms given by the Cheeger-Gromov-Hamilton convergence (see \cite[Chapter 4]{ricciflowtechniques1}). 
\\We first provide a topological characterization of the limit manifold. 
\begin{lemma}\label{exhaustionlemma}
Let $(M_{\infty},g_{\infty}(t),p_{\infty})_{-\infty < t \leq 0}$ be a singularity model for a Ricci flow solution $(\R^{4},g(t))_{0\leq t < T}$ starting at some $g_{0}\in\G$. Then $\pi_{1}(M_{\infty}) = 0.$  
\end{lemma}
\begin{proof} The proof follows from adapting the argument in \cite[Lemma 4.1]{work} which extends to the SU(2)-invariant case due to \eqref{normsquaredcurvatureboundedbyc}.
\end{proof}
Next, we  prove that the symmetries of the flow are enhanced when dilating. To this aim, we first show that the Milnor frame passes to the singularity model. Since the proof of that relies on an Ascoli-Arzel\`{a} argument, we need $C^{3}$-bounds with respect to the rescaled solutions.
\begin{lemma}\label{C3bounds}
There exists a continuous function $f:(-\infty,0]\times (0,+\infty)\rightarrow \mathbb{R}_{\geq 0}$ such that 
\[
\sup_{B_{g_{j}(t)}(p_{j},\nu)}\sum_{k = 0}^{3}\lvert \nabla_{g_{j}(t)}^{k}X_{i}\rvert_{g_{j}(t)} \leq f(t,\nu),
\]
\noindent for any $t\in (-\infty,0]$ and $\nu>0$ and for $i = 1,2,3$. Furthermore, there exists $\alpha > 0$ such that for $i = 1,2,3$ 
\begin{equation}\label{C1lowerbound}
\lvert \nabla_{g_{j}(0)} X_{i} \rvert_{g_{j}(0)}(p_{j}) \geq \alpha,
\end{equation}
\noindent up to passing to a subsequence.
\end{lemma}
\begin{proof}
We fix $t = 0$ and $\nu > 0$ and we let $q\in B_{g_{j}(0)}(p_{j},\nu)$. 
In the following we only analyse the case of $X_{1}$ since the others are proved similarly. We deal with the bounds for $\lvert \nabla_{g_{j}(0)}^{k}X_{1}\rvert_{g_{j}(0)}$ with $k=0,1,2,3$ separately.

 \emph{Case} $k = 0$. We consider a $g(t_{j})$-unit speed geodesic from $p_{j}$ to $q$. From Lemma \ref{firstderivativesbounded} we get
\[
\sqrt{\lambda_{j}}(b(q,t_{j}) - b(p_{j},t_{j})) \leq \sqrt{\lambda_{j}}\left(\sup_{B_{g(t{_j})}(p_{j},\frac{\nu}{\sqrt{\lambda_{j}}})} \lvert \bb_{s} \rvert \right)d_{g(t_{j})}(p_{j},q) \leq \alpha\,\nu.
\]
\noindent The desired estimate then follows from \eqref{normsquaredcurvatureboundedbyc} which gives $\lambda_{j}\bb^{2}(p_{j},t_{j}) \leq \alpha$. 

 \emph{Case} $k=1$. By direct computation we get
\[
\lvert \nabla_{g(t)} X_{1} \rvert^{2}(\cdot,t) = 2\left(\bb_{s}^{2} + 2\frac{\bb^{2}}{\cc^{2}} + \frac{\cc^{2}}{\bb^{2}} - 2 \right)(\cdot,t),
\]
\noindent for any $t\in [0,T)$. Lemma \ref{consistencyRFassumption1} and Lemma \ref{firstderivativesbounded} imply that $\lvert \nabla_{g(t)} X_{1} \rvert(\cdot,t)$ is uniformly bounded and that the estimate \eqref{C1lowerbound} is satisfied.

\emph{Case} $k=2$. We analyse in detail only one exemplificative instance.  
One of the terms appearing in the computation of the norm $(\lambda_{j})^{-\frac{1}{2}}\lvert \nabla_{g(t_{j})}^{2}X_{1}\rvert$ is 
\[
(\lambda_{j})^{-\frac{1}{2}}\lvert\nabla_{g(t_{j})}^{2}X_{1}(\partial s,\partial s, \sigma_{1})\rvert \bb (q, t_{j})\equiv (\lambda_{j})^{-\frac{1}{2}}\lvert \bb_{ss} \rvert(q, t_{j}) \leq \alpha\sqrt{\lambda_{j}}\bb(q,t_{j}),
\]
\noindent where we have used \eqref{sectionalhorizontal01}. The last term is then bounded because it coincides with the case $k=0$ we have already discussed.

 \emph{Case} $k=3$. One of the terms appearing in the computation of $(\lambda_{j})^{-1}\lvert \nabla_{g(t_{j})}^{3}X_{1}\rvert$ is 
\begin{equation}\label{thirdorder}
(\lambda_{j})^{-1}\bb\lvert\nabla_{g(t_{j})}^{3}X_{1}(\partial_{s},\partial_{s},\partial_{s},\sigma_{1})\rvert  (q, t_{j}) = (\lambda_{j})^{-1}\bb\left \vert \frac{\bb_{sss}}{\bb}\right \vert (q, t_{j}).
\end{equation}
\noindent According to Shi's first derivative estimate the covariant derivatives of the curvature are bounded on the singularity models, therefore there exists a uniform constant $\alpha$ such that 
\[
\lvert \nabla_{g(t_{j})} \text{Rm}_{g(t_{j})}\rvert_{g(t_{j})} \leq \alpha (\lambda_{j})^{\frac{3}{2}}.
\]
\noindent Thus we have 
\[
\left \vert \frac{\bb_{sss}}{\bb}\right \vert(q,t_{j}) \leq \left(\left \vert \frac{\bb_{ss}\bb_{s}}{\bb^{2}}\right \vert + \left \vert (k_{01})_{s}\right\vert\right)(q,t_{j}) \leq \left(\alpha\frac{\lambda_{j}}{\bb} + \alpha (\lambda_{j})^{\frac{3}{2}}\right)(q,t_{j}).
\]
\noindent We can then bound the right hand side of \eqref{thirdorder} as 
\[
(\lambda_{j})^{-1}\bb\left \vert \frac{\bb_{sss}}{\bb}\right \vert(q, t_{j}) \leq \alpha(1 + \sqrt{\lambda_{j}}\bb)(q,t_{j}) \leq f(\nu), 
\]
\noindent where the last inequality follows again from the case $k = 0$. The other terms are dealt with similarly.
\\Let now $t\in (-\infty,0]$. By Corollary \ref{cadmitslimit} we get
\[
\lambda_{j}\left \vert \bb^{2}(p_{j},t_{j}) - \bb^{2}(p_{j},t_{j} + \frac{t}{\lambda_{j}}) \right \vert \leq \alpha\lambda_{j}\left \vert \frac{t}{\lambda_{j}} \right \vert \leq \alpha \lvert t \rvert.
\]
\noindent We may then extend the proof of the bound for the case $k = 0$ for any $t\in (-\infty,0]$. The cases $k=1,2,3$ generalize easily.
\end{proof}
Since the rescaled Ricci flows converge to the limit ancient flow in the pointed Cheeger-Gromov sense, from Lemma \ref{C3bounds} it follows that the sequence $(\Phi_{j}^{-1})_{\ast}X_{1}$ is uniformly $C^{3}$-bounded in $B_{g_{\infty}(0)}(p_{\infty},1)$ with respect to $g_{\infty}(0)$. We can then apply the Ascoli-Arzel\`{a} theorem and obtain the following 
\begin{corollaryy}\label{convergenceY1inB1}
There exists a subsequence $(\Phi_{j}^{-1})_{\ast}X_{1}$ that converges in $C^{2}$ to a vector field $X_{1,\infty}$ on $B_{g_{\infty}(0)}(p_{\infty},1)$.
\end{corollaryy}
From now on we re-index the subsequence given by the previous Corollary. 
In order to prove that $X_{1,\infty}$ is actually a Killing vector field for $g_{\infty}(0)$ we need a preliminary result. The following shows that the singularity model cannot be Ricci flat. 
\begin{lemma}\label{bphihgoesto0true}
For any $q\in M_{\infty}$ and for any $t\in (-\infty,0]$ the following is satisfied: 
\[
\lim_{j\rightarrow \infty}\bb(\Phi_{j}(q),t_{j} + \frac{t}{\lambda_{j}}) = 0. 
\]
\end{lemma}
\begin{proof}
Suppose for a contradiction that there exist $q\in M_{\infty}$, $t\in (-\infty,0]$, a subsequence (which we still denote by $j$) and $\mu > 0$ such that $\bb(\Phi_{j}(q),t_{j} + (\lambda_{j})^{-1}t)\rightarrow \mu$. By \eqref{normsquaredcurvatureboundedbyc} we immediately derive that $R_{g_{\infty}(t)}(q) = 0$. 
Since any complete ancient solution to the Ricci flow has nonnegative scalar curvature \cite{chen2009}, a standard application of the maximum principle and the uniqueness of the flow among complete and bounded curvature solutions yield $\text{Ric}_{\infty} \equiv 0$ everywhere in the space-time. We then fix the time to be $0$ and assume that $g_{\infty}(0)$ is not flat. By the uniform $C^{1}$-lower bound in \eqref{C1lowerbound} there exists an open subset $U\subset B_{g_{\infty}(0)}(p_{\infty},1)$ where $\lvert X_{1,\infty}\rvert_{g_{\infty}(0)}|_{U} > 0$, with $X_{1,\infty}$ given by Corollary \ref{convergenceY1inB1}. From the real analyticity of the ancient limit flow \cite{bando} it follows that there exists $\bar{q}\in U$ such that $\lvert \text{Rm}_{g_{\infty}(0)}\rvert_{g_{\infty}(0)}(\bar{q}) > 0$, otherwise the limit would be flat. Moreover, $\bb(\Phi_{j}(\bar{q}),t_{j}) \rightarrow 0$ as $j\rightarrow \infty$. For if such condition did not hold, then by \eqref{normsquaredcurvatureboundedbyc} the Riemann tensor would vanish at $\bar{q}$. Since $g_{\infty}(0)$ is Ricci flat we get 
\begin{align*}
0 &= \left \vert \text{Ric}_{g_{\infty}(0)} \right \vert^{2}_{g_{\infty}(0)}(\bar{q}) \notag \\ &=\lim_{j\rightarrow \infty}\frac{1}{\lambda^{2}_{j}}\left((k_{01} + k_{02} + k_{03})^{2} + 2 (k_{01} + k_{12} + k_{13})^{2} + (k_{03} + k_{13} + k_{23})^{2}\right)(\Phi_{j}(\bar{q}),t_{j}) \notag \\ &= \lim_{j\rightarrow \infty}\frac{1}{\lambda^{2}_{j}\bb^{4}}\left(\bb^{4}\left((2k_{01} + k_{03})^{2} + 2 (k_{01} + k_{12} + k_{13})^{2} + (k_{03} + 2k_{13})^{2}\right)\right)(\Phi_{j}(\bar{q}),t_{j}).  
\end{align*}
\noindent By the Cheeger-Gromov convergence we get $\Phi_{j}(\bar{q})\in B_{g_{j}(0)}(p_{j},2)$ for $j$ large enough. Therefore, from Lemma \ref{C3bounds} (the case of $k=0$) it follows that $\lambda_{j}\bb^{2}(\Phi_{j}(\bar{q}),t_{j})\leq \alpha$ for $j$ large and for some positive $\alpha$. From the estimate in Lemma \ref{rotationalsymmetry2} we finally derive that the limit above is zero if and only if
\[
\lim_{j\rightarrow \infty}\left(\bb^{2}\lvert \text{sec}_{g(t_{j})}\rvert\right)(\Phi_{j}(\bar{q}),t_{j}) = 0,
\]
\noindent with $\text{sec}_{g(t_{j})}$ the maximal sectional curvature of $g(t_{j})$. 
Therefore by Corollary \ref{convergenceY1inB1} and the choice of $\bar{q}$, up to passing to a diagonal subsequence, we conclude that
\begin{align*}
0 < \lvert X_{1,\infty}\rvert_{g_{\infty}(0)}^{2}\lvert \text{Rm}_{g_{\infty}(0)}\rvert_{g_{\infty}(0)}(\bar{q}) &=  \lim_{j\rightarrow \infty}\lvert X_{1}\rvert_{g_{j}(0)}^{2}\lvert \text{Rm}_{g_{j}(0)}\rvert_{g_{j}(0)}(\Phi_{j}(\bar{q})) \\ &= \lim_{j\rightarrow\infty}\left(\bb^{2}\lvert \text{Rm}_{g(t_{j})}\rvert_{g(t_{j})}\right)(\Phi_{j}(\bar{q})),
\end{align*}
\noindent which is a contradiction because we have just proved that the right hand side must vanish. 
\end{proof}
We can now show that $X_{1,\infty}$ is a Killing vector field on the limit manifold for any time. 
\begin{lemma}\label{Yi'sKilling}
There exists a unique smooth extension of $X_{1,\infty}$ to the limit manifold $M_{\infty}$ such that $(\Phi_{j}^{-1})_{\ast}X_{1}$ converges in $C^{2}$ to $X_{1,\infty}$ on compact sets. Moreover $X_{1,\infty}$ is a $g_{\infty}(t)$-Killing vector field for any $t\in (-\infty,0]$. 
\end{lemma}
\begin{proof}
We first prove that $X_{1,\infty}$ is a Killing vector field in $B_{g_{\infty}(0)}(p_{\infty},1)$ with respect to $g_{\infty}(0)$. Suppose for a contradiction that there exist $q\in B_{g_{\infty}(0)}(p_{\infty},1)$, $\delta > 0$ and $Z,W\in C^{\infty}(TM_{\infty})$ such that
\[
g_{\infty}(0)\left(\nabla^{g_{\infty}(0)}_{Z}X_{1,\infty}, W\right) + g_{\infty}(0)\left(Z, \nabla^{g_{\infty}(0)}_{W}X_{1,\infty} \right) \geq \delta > 0
\]
\noindent in some compact neighbourhood $\overline{\Omega}$ of $q$. By Corollary \ref{convergenceY1inB1} and the Cheeger-Gromov convergence we get
\[
\left(g_{j}(0)\left(\nabla^{g_{j}(0)}_{(\Phi_{j})_{\ast}Z}X_{1}, (\Phi_{j})_{\ast}W\right) + g_{j}(0)\left((\Phi_{j})_{\ast}Z, \nabla^{g_{j}(0)}_{(\Phi_{j})_{\ast}W}X_{1} \right)\right)\left(\Phi_{j}(q)\right) \geq \frac{\delta}{3},
\]
\noindent for some $j$ large enough. If at $\Phi_{j}(q)$ we write $(\Phi_{j})_{\ast}Z = z^{0}_{j}\partial_{s(t_{j})} + \sum_{k=1}^{3}z_{j}^{k}X_{k}$ and similarly for $(\Phi_{j})_{\ast}W$, then by the Koszul formula we get
\[
\frac{\delta}{3}\leq \left(2\lambda_{j}\bb^{2}\left(1 - \frac{\cc^{2}}{\bb^{2}}\right)\left(z_{j}^{2}w_{j}^{3} + z_{j}^{3}w_{j}^{2}\right)\right)\left(\Phi_{j}(q)\right).
\]
\noindent Since $\lvert Z\rvert_{g_{\infty}(0)}(q,0)\geq \lim_{j\rightarrow \infty}\sqrt{\lambda_{j}}\lvert z_{j}^{k}\rvert\bb(\Phi_{j}(q),t_{j})$, for $k = 1,2$ and similarly for $W$, we can use Lemma \ref{consistencyRFassumption1} (with $\varepsilon > 0$) for the case $k=3$ and conclude that there exists a positive constant $\beta$ depending on $q$ such that
\[
\frac{\delta}{3}\leq 2\beta\left(1 - \frac{\cc^{2}}{\bb^{2}}\right)\left(\Phi_{j}(q),t_{j}\right) \leq \alpha \bb\left(\Phi_{j}(q),t_{j}\right),
\]
\noindent where we have used Lemma \ref{rotationalsymmetry}. According to Lemma \ref{bphihgoesto0true} we can choose $j$ sufficiently large such that the right hand side is as small as we need, thus obtaining the contradiction.
\\Since the limit ancient flow is real analytic \cite{bando} and by Lemma \ref{exhaustionlemma} $M_{\infty}$ is simply connected, it is a classic result that $X_{1,\infty}$ extends uniquely to a global Killing vector field on $(M_{\infty},g_{\infty}(0))$ \cite{nomizu}. Being $g_{\infty}(0)$ complete, we also get that $X_{1,\infty}$ is smooth. 
\\Given $\nu > 1$, Lemma \ref{C3bounds} implies that for any subsequence of $(\Phi_{j}^{-1})_{\ast}X_{1}$ there exists a subsubsequence that converges in $C^{2}$ to some vector field on $B_{g_{\infty}(0)}(p_{\infty},\nu)$. The argument above shows that the limit vector field must be a Killing field for $g_{\infty}(0)$. By the uniqueness result in \cite{nomizu} we conclude that such limit vector field is indeed $X_{1,\infty}$. The statement is then proved when $t = 0$.
The very same proof for the case $t=0$ works when $t\in (-\infty,0]$. 
\end{proof}
The lower bound \eqref{C1lowerbound} and the previous Lemma extend to the sequences $(\Phi_{j}^{-1})_{\ast}X_{2}$ and $(\Phi_{j}^{-1})_{\ast}X_{3}$ which then define analogous Killing vector fields $X_{2,\infty}$ and $X_{3,\infty}$ for the singularity model. Moreover, from the Cheeger-Gromov-Hamilton convergence we derive that the system $\{X_{i,\infty}\}_{i=1}^{3}$ is an orthogonal frame with respect to $g_{\infty}(t)$ for any $t\in (-\infty,0]$. We can now prove that this frame of Killing fields implies that the singularity model is spherically symmetric.
\begin{lemma}\label{rotationalsymmetrylimit}
The metric $g_{\infty}(t)$ is rotationally symmetric for any $t\in (-\infty,0]$. Moreover $M_{\infty} = \mathbb{R}^{4}$ or $M_{\infty} = \mathbb{R}\times S^{3}$.
\end{lemma}
\begin{proof}
According to \eqref{C1lowerbound} and the orthogonality of the vector fields $X_{i,\infty}$ there exists at least a point $q\in M_{\infty}$ where this frame spans a 3-dimensional subspace of $T_{q}M_{\infty}$. Therefore, since the Lie brackets are preserved in the limit, Lemma \ref{Yi'sKilling} implies that there exists a (non-trivial) copy of $\mathfrak{su}(2)$ in the Lie algebra of Killing fields $\mathfrak{iso}(M_{\infty},g_{\infty}(t))$. By integrating the Killing fields we derive that SU(2) acts isometrically with cohomogeneity 1 on $(M_{\infty},g_{\infty}(t))$ for any $t\in (-\infty,0]$. In particular, by the Lie algebra constants we see that $\{X_{i,\infty}\}_{i=1}^{3}$ is a Milnor frame for $g_{\infty}(t)$. 
\\By the classification of connected non-compact manifolds supporting the cohomogeneity 1 action of a compact Lie group there exists at the most one singular orbit $\mathcal{O}_{sing}$ for the SU(2) action on $M_{\infty}$ \cite{cohomogeneity1}. Moreover, we can write $M_{\infty} = \mathcal{O}_{sing} \cup M_{\text{prin}}$, where $M_{\text{prin}}$ is an open dense submanifold foliated by maximal orbits of the form
\begin{equation}\label{definitionprincipalpart}
M_{\text{prin}} = \mathbb{R}\times \text{SU(2)}/H,
\end{equation}
\noindent with $H$ the isotropy group of the action along principal orbits \cite{cohomogeneity1}. We note that when $\mathcal{O}_{sing} = \emptyset$ Lemma \ref{exhaustionlemma} implies that $M_{\infty} = \mathbb{R}\times S^{3}$.
\\All the information about $g_{\infty}(t)$ can be obtained by restricting it to a geodesic starting at the singular orbit and meeting the principal orbits orthogonally. Namely, once we denote the dual coframe associated with $\{X_{i,\infty}\}_{i=1}^{3}$ by $\{\sigma_{i,\infty}\}_{i=1}^{3}$, we have 
\[
g_{\infty}(t)|_{M_{\text{prin}}} = (dy_{t})^{2} + \phi_{1,\infty}^{2}(y,t)\,\sigma_{1,\infty}^{2} + \phi_{2,\infty}^{2}(y,t)\,\sigma_{2,\infty}^{2} + \phi_{3,\infty}^{2}(y,t)\,\sigma_{3,\infty}^{2},
\]
\noindent with $\phi_{i,\infty}(y,t) \doteq \lvert X_{i,\infty}\rvert_{g_{\infty}(t)} (y)$ for any $y > 0$, $t\leq 0$ and $i = 1,2,3$. 
\\Let $q\in M_{\text{prin}}$. By the convergence of $(\Phi_{j}^{-1})_{\ast}X_{i}$ to $X_{i,\infty}$ on compact sets, we get 
\begin{align*}
\frac{1}{\phi_{1,\infty}}\left(\frac{\phi_{1,\infty}}{\phi_{3,\infty}}-1\right)(q,t) &\equiv \frac{1}{\lvert X_{1,\infty}\rvert_{g_{\infty}(t)}}\left(\frac{\lvert X_{1,\infty}\rvert_{g_{\infty}(t)}}{\lvert X_{3,\infty}\rvert_{g_{\infty}(t)}}-1\right)(q) \\ &= \lim_{j\rightarrow \infty} \frac{1}{\sqrt{\lambda_{j}}\bb}\left(\frac{\bb}{\cc} - 1\right)\left(\Phi_{j}(q),t_{j} + \frac{t}{\lambda_{j}}\right) \leq 0,
\end{align*}
\noindent where we have used the estimate in Lemma \ref{rotationalsymmetry}. Since the ratio $\bb/\cc\geq 1$ is scale invariant we obtain $\phi_{1,\infty} = \phi_{3,\infty}$. The final identity $\phi_{1,\infty} = \phi_{2,\infty}$ is a consequence of the extra U(1)-symmetry. Thus $g_{\infty}(t)$ is rotationally symmetric. Furthermore, if there exists a singular orbit, then $\phi_{i,\infty} \equiv \phi_{\infty}$ is an odd function with $\partial_{y_{t}}\phi_{\infty}(\mathcal{O}_{sing},t) = 1$ (see, e.g., \cite{cohomogeneity1}). From the boundary conditions \eqref{smoothnessorigin} we deduce that $M_{\infty} = \mathbb{R}^{4}$. We may finally conclude that $M_{\infty}=\mathbb{R}^{4}$ or $M_{\infty} = \mathbb{R}\times S^{3}$ with 
\begin{equation}\label{sphericalsymmetrylimitformula}
g_{\infty}(t) = (dy_{t})^{2} + \phi_{\infty}^{2}(y,t)\bar{g}_{S^{3}},
\end{equation}
\noindent where $\bar{g}_{S^{3}}$ is the standard constant curvature 1 metric on $S^{3}$ and 
\[
\phi_{\infty}(q,t) = \lim_{j\rightarrow\infty}\sqrt{\lambda_{j}}\,\bb\left(\Phi_{j}(q),t_{j}+ \frac{t}{\lambda_{j}}\right),
\] 
\noindent for any $(q,t)\in M_{\infty}\times (-\infty,0]$.
\end{proof}
We now show that Lemma \ref{rotationalsymmetrylimit} actually extends to any singularity model of a warped Berger Ricci flow in $\Gin$. Indeed, given a blow-up sequence $(p_{j},t_{j})$ as above and a radial coordinate $x_{1} > \rho$, with $\rho$ satisfying \eqref{definitionradiusrho}, then by Lemma \ref{estimatesforGin} it suffices to prove that any rescaled geodesic ball $B_{g_{j}(t)}(p_{j},\nu)$ lies in $B(\origin,x_{1})$ for $j$ sufficiently large.
\begin{lemma}\label{claimzozzo}
Let $(M_{\infty},g_{\infty}(t),p_{\infty})_{-\infty < t \leq 0}$ be a singularity model for a warped Berger Ricci flow $(\R^{4},g(t))_{0\leq t < T}$, with $T < \infty$, starting at some $g_{0}\in \Gin$. For any $t\leq 0$, for any $\nu > 0$ and for any $x_{1} > \rho$, with $\rho$ satisfying \eqref{definitionradiusrho}, there exists $j_{0} = j_{0}(t,\nu,x_{1})$ such that for all $j\geq j_{0}$ we have
\[
B_{g_{j}(t)}(p_{j},\nu)\subset B(\origin,x_{1}).
\]
\end{lemma} 
\begin{proof}
Given a blow-up sequence $(p_{j},t_{j})$ with $p_{j} = (x_{j},\theta)$ for some $\theta\in S^{3}$, we observe that up to a finite number of indices we have $x_{j}< \rho$ otherwise $\lvert \text{Rm}_{g(t_{j})}\rvert_{g(t_{j})}(p_{j})$ would be bounded. Suppose for a contradiction that there exist a time $t$, a radius $\nu$, a coordinate $x_{1} > \rho$ and a subsequence $q_{j_{k}} = (y_{j_{k}},\theta_{j_{k}})\subset B_{g_{j}(t)}(p_{j},\nu)$ such that $y_{j_{k}} > x_{1}$. Then by \eqref{definitionradiusrho} and standard distortion estimates of the Riemannian distance we get
\begin{align*}
\frac{\nu}{\sqrt{\lambda_{j_{k}}}} &> d_{g(t_{j_{k}} + \frac{t}{\lambda_{j_{k}}})}(p_{j_{k}},q_{j_{k}}) \geq \inf_{\substack{y\in S(\origin,\rho) \\ z\in S(\origin,x_{1})}}d_{g(t_{j_{k}} + \frac{t}{\lambda_{j_{k}}})}(y,z) \\ &\geq 
\alpha \inf_{\substack{y\in S(\origin,\rho) \\ z\in S(\origin,x_{1})}}d_{g_{0}}(y,z), 
\end{align*}
\noindent which then gives us a contradiction for $k$ large enough.
\end{proof}
We may then adapt all the arguments above and conclude that any singularity model of a warped Berger Ricci flow in $\Gin$ is rotationally symmetric. We finally address the proof of the classification result in Proposition \ref{classificationsingularity}.
\begin{proof}[Proof of Proposition \ref{classificationsingularity}] Lemma \ref{rotationalsymmetrylimit} implies that any singularity model is in particular conformally flat. Thus by \cite{zhang} we derive that any singularity model has nonnegative curvature operator. Since we have shown that singularity models are weakly $\kappa$-non-collapsed at all scales, we find that any singularity model is a $\kappa$-solution to the Ricci flow. If the curvature operator is not strictly positive at some point in the space-time, then the same argument in \cite[Lemma 4.3]{work} shows that $(M_{\infty},g_{\infty}(t))$ splits off a line and must hence be isometric to the self-similar shrinking soliton on the cylinder $\R\times S^{3}$.
\\Conversely, if the curvature operator is strictly positive at a point, then by the strong maximum principle we conclude that the singularity model is positively curved. 
\end{proof}
\begin{remark}\label{classificationextends}
We point out that Proposition \ref{classificationsingularity} extends to warped Berger Ricci flows for which both the estimate \eqref{normsquaredcurvatureboundedbyc} and the rotational symmetry type of bounds in Lemmas \ref{rotationalsymmetry}, \ref{rotationalsymmetryorder1} and \ref{rotationalsymmetry2} are satisfied.  
\end{remark}


\section{Proofs of the main results}
\subsection{Bryant soliton singularities.}
In this subsection we show that any Ricci flow in $\G$ encounters a Type-II singularity and that Theorem \ref{maintheoremtype2bis} is satisfied.
\begin{proof}[Proof of Theorem \ref{maintheoremtype2}]
According to Lemma \ref{finite-timesingularity} the Ricci flow develops a finite-time singularity at some $T<\infty$. Suppose that the Ricci flow is Type-I and let $\Sigma$ be the singular set defined as in \cite[Definiton 1.5]{type1}. 

If the origin $\origin$ does not belong to $\Sigma$, then the flow stays smooth on $B(\origin,2r)$ for some $r > 0$. Thus there exists $\delta > 0$ such that $\bb(r,t)\geq \delta > 0$ for any $t\in [0,T)$. Lemma \ref{cruciallemma} then implies $\bb(x,t)\geq \delta$ for any $x\geq r$ and for all $t\in [0,T)$. From the estimate \eqref{normsquaredcurvatureboundedbyc} we finally deduce that the curvature stays uniformly bounded outside $B(\origin,r)$ and hence on $\R^{4}$. The latter condition contradicts that the flow develops a singularity at $T$ \cite{shi}. 

If $\origin\in\Sigma$, then we can apply \cite[Theorem 1.1]{type1} and derive that any parabolic dilation of the flow at $\origin$ (sub)converges to a \emph{non-flat} shrinking soliton. By the classification in Proposition \ref{classificationsingularity} we get that any such singularity model is a shrinking cylinder \cite{kotschwar}. However by the SU(2) symmetry and the Cheeger-Gromov convergence we have just shown that the cylinder $\R\times S^{3}$ is exhausted by open sets diffeomorphic to $\R^{4}$, which is not possible. Therefore, the singularity is Type-II.

Since the flow is Type-II and non-collapsed we can choose a blow-up sequence giving rise to a singularity model which consists of an \emph{eternal} solution \cite[Section 16]{formationsingularities}). By the classification in Proposition \ref{classificationsingularity} we deduce that such eternal solution is rotationally symmetric and \emph{positively} curved\footnote{At this point, one can also rely on the recent classification of rotationally symmetric $\kappa$-solutions in \cite{brendle2} to conclude that such eternal solution must be isometric to the Bryant soliton. However, we chose to present a more self-contained argument which is sufficient to complete the proof of Theorem 1.}. Therefore the scalar curvature and the Riemann curvature are comparable up to the singular time and we can hence adapt the argument in \cite{formationsingularities} to extract a space-time sequence $(p_{j},t_{j})$, with $t_{j}\nearrow T$, such that if we set $\lambda_{j} \doteq R_{g(t_{j})}(p_{j})$, then the rescaled Ricci flows $(\R^{4},g_{j}(t),p_{j})$ defined by $g_{j}(t)\doteq \lambda_{j}g(t_{j} + (\lambda_{j})^{-1}t)$ (sub)converge in the pointed Cheeger-Gromov sense to a $\kappa$-solution whose scalar curvature attains its supremum in the space-time. According to \cite{eternal} the singularity model is then a gradient steady soliton and must hence be isometric to the Bryant soliton by the classification in Proposition \ref{classificationsingularity}.   
\end{proof}
\vspace{0.2in}
\begin{proof}[Proof of Theorem \ref{maintheoremtype2bis}] We prove the four points in Theorem \ref{maintheoremtype2bis} separately.

\emph{(i) The Bryant soliton appears at the origin.} Let $(p_{j},t_{j})$ and $\lambda_{j}$ be defined as in the proof of Theorem \ref{maintheoremtype2}, let $\Phi_{j}$ be the family of diffeomorphisms given by the Cheeger-Gromov convergence and let $(\R^{4},g_{\infty}(t),p_{\infty})$ be the Bryant soliton arising as limit singularity model. By the SU(2) symmetry we may choose $p_{j}$ of the form $(x_{j},\theta)$ for some $\theta\in S^{3}$. Suppose for a contradiction that there exists $\delta > 0$ such that for $j$ sufficiently large we have $d_{g_{j}(0)}(\origin,p_{j}) \geq 2\delta > 0$. We may then find points $q_{j} \equiv (\tilde{x}_{j},\theta)$, with $\tilde{x}_{j} < x_{j}$, such that, up to passing to a subsequence, $\Phi_{j}^{-1}(q_{j})\rightarrow q_{\infty}$ for some $q_{\infty}$ satisfying $d_{g_{\infty}(0)}(p_{\infty},q_{\infty}) = \delta$. Since the scalar curvature of the Bryant soliton $g_{\infty}(t)$ attains its maximum at the centre of symmetry, i.e. at the origin of $\R^{4}$, we deduce that $p_{\infty} = \origin\in \R^{4}$ and therefore that the Killing vectors $\{X_{i,\infty}\}$ constructed above need to vanish at $p_{\infty}$. Equivalently, from the argument in Lemma \ref{rotationalsymmetrylimit} we derive that
\[
0 = \lvert X_{1,\infty} \rvert_{g_{\infty}(0)}(p_{\infty}) = \lim_{j\rightarrow\infty} \sqrt{\lambda_{j}}\bb(x_{j},t_{j}).
\]
\noindent Since the warping coefficient $\bb$ is monotone in space and $\tilde{x}_{j} < x_{j}$ we have 
\[
\lvert X_{1,\infty}\rvert_{g_{\infty}(0)}(q_{\infty}) = \lim_{j\rightarrow\infty}\sqrt{\lambda_{j
}}\bb(\tilde{x}_{j},t_{j}) \leq 0,
\]
\noindent which is not possible because the Killing fields generating the rotational symmetry cannot vanish along a principal orbit, i.e. away from the origin. Therefore, up to choosing a subsequence, we have $d_{g_{j}(0)}(\origin,p_{j}) \rightarrow 0$. In particular, we may pick a subsequence such that $R_{g(t_{j})}(\origin) \geq (1-\delta_{j})\lambda_{j}$ for some $\delta_{j}\rightarrow 0$. If we then dilate the Ricci flow by factors $R_{g(t_{j})}(\origin)$ we still obtain the Bryant soliton as pointed Cheeger-Gromov limit.

\emph{(ii) The singularity is global.} Consider the set of points where the flow becomes singular as $t\nearrow T$:
\[
\Omega \doteq \left\{p\in\R^{4}: \lim_{t\nearrow T}\bb(p,t) = \lim_{t\nearrow T}\cc(p,t) = 0. \right\}.
\]
\noindent We note that the previous definition makes sense due to Corollary \ref{cadmitslimit} and the estimate \eqref{normsquaredcurvatureboundedbyc}. Part (ii) in the statement of Theorem \ref{maintheoremtype2bis} is equivalent to showing $\Omega = \R^{4}$. Indeed we have proved above that the curvature cannot stay uniformly bounded at the origin, while away from the origin the estimate \eqref{normsquaredcurvatureboundedbyc} implies that both $\bb$ and $\cc$ need to converge to zero as $t\nearrow T$ for the curvature to blow-up. 
We assume for a contradiction that $\Omega \neq \R^{4}$. By Lemma \ref{cruciallemma} there exists $\bar{x} \geq 0$ satisfying $\Omega = B(\origin,\bar{x})$. We may always take the Euclidean ball $B(\origin,\bar{x})$ to be \emph{closed} because by Corollary \ref{cadmitslimit} there exists a uniform constant $\alpha > 0$ such that $\bb^{2}(x,t) \leq \alpha(T-t)$ for all $x < \bar{x}$.
\begin{claim}\label{claimderivativezero}
Let $(\R^{4},g(t))_{0\leq t < T}$ be the Ricci flow starting at some $g_{0}\in\G$. Then 
$\lim_{t\nearrow T}\cc H(x,t) = 0$ for any $x > \bar{x}$.
\end{claim}
\begin{proof}[Proof of Claim \ref{claimderivativezero}]
We prove the Claim by a blow-up argument. Namely, we show that if the statement was false, then any singularity model would have Euclidean volume growth, thus leading to a contradiction.
\\Since by \eqref{normsquaredcurvatureboundedbyc} the curvature is uniformly controlled in time for any radius $x > \bar{x}$ by some positive constant only depending on $x$, the same argument in the proof of Lemma \ref{definitionxtilde} shows that the limit $\lim_{t\nearrow T}\cc H(x,t)$ is well defined and finite (and nonnegative by Lemma \ref{cruciallemma}) for any $x > \bar{x}$.
Suppose that there exists $x_{0} > \bar{x}$ such that $\lim_{t\nearrow T}\cc H(x_{0},t) > 0$. Then the same argument in Corollary \ref{lowerboundH} implies that $\cc H$ is uniformly bounded from below by some $\mu > 0$ on the ball $B(\origin,x_{0})$ for times close to $T$. In particular, by Lemma \ref{estimatesforGin} there exists $\alpha > 0$ satisfying
\[
\bb_{s} \geq \frac{1}{3}\left(\mu \frac{\bb}{\cc} - \alpha\bb\right)
\]
\noindent on $B(\origin,x_{0})$ for times close enough to $T$. Let us rescale the solution along a blow-up sequence and let $(M_{\infty},g_{\infty}(t),p_{\infty})$ be the associated singularity model. We note that by Lemma \ref{bphihgoesto0true} if $q\in M_{\infty}$, then $\Phi_{j}(q)\in B(\origin,x_{0})$ for $j$ large enough. Moreover, we have 
\[
\bb_{s}(\Phi_{j}(q),t_{j}) \geq \frac{1}{3}\left(\mu \frac{\bb}{\cc} - \alpha\bb\right)(\Phi_{j}(q),t_{j}) \geq \frac{1}{6}\mu.
\]
\noindent Thus, from Corollary \ref{convergenceY1inB1} and Lemma \ref{rotationalsymmetrylimit} we derive that
\begin{align*}
(1 - (\partial_{y}\phi_{\infty})^{2})(q,0) &= (\phi_{\infty}^{2}k_{12}^{\infty})(q,0) = \lim_{j\rightarrow \infty}\left(4 - \bb_{s}^{2} - 3\frac{\cc^{2}}{\bb^{2}}\right)(\Phi_{j}(q),t_{j}) \\ &= \lim_{j\rightarrow \infty}\left(1 - \bb_{s}^{2}\right)(\Phi_{j}(q),t_{j}) \leq 1 - \frac{1}{36}\mu^{2},
\end{align*}
\noindent which then implies $\partial_{y}\phi_{\infty}(q,0) \geq \mu/6$ for any $q\in M_{\infty}$. Since the limit is rotationally symmetric we obtain
\begin{equation}\label{positiveAVReqn}
\text{Vol}_{g_{\infty}(0)}(B_{g_{\infty}(0)}(p_{\infty},r)) \geq \alpha r^{4},
\end{equation}
\noindent for any $r\geq 1$ and for some $\alpha > 0$. By Proposition \ref{classificationsingularity} and the bound \eqref{positiveAVReqn} we conclude that $(M_{\infty},g_{\infty}(t))$ is a non-compact $\kappa$-solution with \emph{positive asymptotic volume ratio}. According to a rigidity property proved by Perelman \cite[Proposition 11.4]{pseudolocality} $g_{\infty}(t)$ must then be flat, which is a contradiction.
\end{proof}
\vspace{0.1in}
Since by Claim \ref{claimderivativezero} $(\bb^{2}\cc)_{s}(x,t) = \bb^{2}\cc H(x,t) \rightarrow 0$ as $t\nearrow T$ for all $x > \bar{x}$ we can argue as in the proof of Lemma \ref{definitionxtilde} and deduce that there exists $\gamma > 0$ such that 
\begin{equation}\label{redefinitiongamma}
\lim_{t\nearrow T}\bb^{2}\cc(x,t)= \gamma,\,\,\,\,\,\,\, \forall x > \bar{x}.
\end{equation}
\noindent We now show that if $\bb$ is small at $\bar{x}$ for times close to $T$, then $\bb$ cannot jump to some positive quantity $\gamma^{1/3}$ for all $x > \bar{x}$ when $t\nearrow T$. Let $\varepsilon < 1$ and $T_{\varepsilon} < T$ to be chosen below such that $\bb^{2}(\bar{x},t) \leq \varepsilon/2$ for all $t\in [T_{\varepsilon},T)$. We let $x_{\varepsilon} > \bar{x}$ be such that $\bb^{2}(x_{\varepsilon},T_{\varepsilon}) \leq \varepsilon$ and $\tilde{T}$ be the first time larger than $T_{\varepsilon}$ such that $\bb(x_{\varepsilon},\tilde{T}) = 1$, if such time exists. By Lemma \ref{bssquaredminus1} and Lemma \ref{estimatelog} we have
\begin{align*}
\partial_{t}\bb^{2}(x_{\varepsilon},t) & = 2\left(\bb\bb_{ss} + \frac{\bb\bb_{s}\cc_{s}}{\cc} + \bb_{s}^{2} + 2\frac{\cc^{2}}{\bb^{2}} - 4\right)(x_{\varepsilon},t) \leq 2(\bb\bb_{ss} + \alpha\bb)(x_{\varepsilon},t) \\ &\leq \alpha\left(\frac{1}{\lvert \text{log}\bb\rvert} + \bb\right)(x_{\varepsilon},t) \leq \frac{\alpha}{\lvert \text{log}\bb\rvert}(x_{\varepsilon},t),
\end{align*}
\noindent for some $\alpha > 0$ independent of $\varepsilon$ and $t$ and for all $t\in [T_{\varepsilon},\tilde{T}]$. Thus, we can integrate the previous inequality and obtain
\[
\bb^{2}(x_{\varepsilon},t)\left(2\lvert\text{log}(\bb(x_{\varepsilon},t))\rvert + 1\right) \leq \alpha (t-T_{\varepsilon}) + \bb^{2}(x_{\varepsilon},T_{\varepsilon})\left(2\lvert\text{log}(\bb(x_{\varepsilon},T_{\varepsilon}))\rvert + 1 \right).
\] 
\noindent Therefore, since $\bb^{2}(x_{\varepsilon},T_{\varepsilon}) \leq \varepsilon < 1$ we get
\[
\bb^{2}(x_{\varepsilon},t) \leq \alpha(T-T_{\varepsilon}) + 3\varepsilon.
\]
\noindent Once we choose $\varepsilon$ and $T_{\varepsilon}$ accordingly, we derive that $\tilde{T}$ does not exist and hence that $\bb^{2}(x_{\varepsilon},t)\leq \gamma^{\frac{2}{3}}/4$ for all $t\in [T_{\varepsilon},T)$. We then find
\[
\bb^{2}\cc(x_{\varepsilon},t) \leq \bb^{3}(x_{\varepsilon},t) \leq \gamma/8,
\]
\noindent which contradicts \eqref{redefinitiongamma}. Therefore $\Omega = \R^{4}$.

\emph{(iii) Type-I blow-up at infinity.} Once we know that the singularity is global it is natural to expect shrinking cylinders to appear when dilating the solution at infinity. 
\\Let $t_{j}\nearrow T$, $\delta > 0$ arbitrary and $\varepsilon > 0$ a positive quantity to be chosen below. Since the spatial derivatives of $\bb$ and $\cc$ are decaying to zero at infinity for any $t\geq T/2$, we may always pick points $p_{j}$ such that $d_{g_{0}}(\origin,p_{j})\rightarrow \infty$ and 
\begin{equation}\label{controlinfinitypj}
\sup_{B_{g(t)}(p_{j},\delta)}\left(\lvert k_{01}\rvert + \lvert k_{03}\rvert + \left\vert\frac{\bb_{s}}{\bb}\right\vert + \left\vert\frac{\cc_{s}}{\cc}\right\vert\right) \leq \varepsilon,
\end{equation}
\noindent for all $t\in[T/2, T_{j}\doteq (T+t_{j})/2]$. Let us denote the factors $R_{g(t_{j})}(p_{j})$ by $\lambda_{j}$. From \eqref{controlinfinitypj} we derive $\lambda_{j}\bb^{2}(p_{j},t_{j})\geq \beta > 0$ uniformly with respect to $j$. Since by Corollary \ref{cadmitslimit} and part (ii) of Theorem \ref{maintheoremtype2bis} $\bb^{2}(\cdot,t) \leq \alpha(T-t)$ for some uniform constant $\alpha > 0$, we also see that $\lambda_{j} \rightarrow \infty$. Similarly, by \eqref{controlinfinitypj} we find that $\partial_{t}\bb^{2}(\cdot,t)\leq -\alpha < 0$ in $B_{g(t)}(p_{j},\delta)$ for all $t\in[T/2, T_{j}]$. Therefore we have
\[
(T-t_{j})\lambda_{j} = 2(T_{j} - t_{j})\lambda_{j} \leq \frac{\alpha}{\bb^{2}(p_{j},t_{j})}(T_{j}-t_{j}) \leq \frac{\alpha(T_{j}-t_{j})}{\bb^{2}(p_{j},T_{j}) + \alpha(T_{j}-t_{j})} \leq \alpha,
\]
\noindent where we have used \eqref{normsquaredcurvatureboundedbyc}. Analogously, given $\nu > 0$, $t\leq 0$ and $p\in B_{g(t_{j})}(p_{j},\nu(\lambda_{j})^{-1/2})$ we see that
\[
\lambda_{j}\bb^{2}(p,t_{j}+ (\lambda_{j})^{-1}t) \geq \lambda_{j}\bb^{2}(p,t_{j}) + \alpha\lvert t\rvert,
\]
\noindent for $j$ large enough. From \eqref{controlinfinitypj} we also derive the following spatial control:
\[
\left\vert\text{log}\left(\frac{\bb(p,t_{j})}{\bb(p_{j},t_{j})}\right)\right\vert \leq \varepsilon\frac{\nu}{\sqrt{\lambda_{j}}}.
\]
\noindent We may finally estimate the curvature of the rescaled Ricci flows as 
\begin{align*}
\frac{1}{\lambda_{j}}\lvert\text{Rm}_{g(t_{j}+(\lambda_{j})^{-1}t)}\rvert(p)&\leq \frac{\alpha}{\lambda_{j}\bb^{2}(p,t_{j}+(\lambda_{j})^{-1}t))} \leq \frac{\alpha}{\lambda_{j}\bb^{2}(p,t_{j}) + \alpha\lvert t\rvert} \\ &\leq \frac{\alpha}{\beta\exp(-2\frac{\varepsilon\nu}{\sqrt{\lambda_{j}}}) + \alpha\lvert t\rvert}.
\end{align*}
\noindent Since the flow is weakly $\kappa$-non-collapsed for some $\kappa > 0$ we may apply Hamilton's compactness theorem and conclude that the sequence of rescaled Ricci flows converge in the pointed Cheeger-Gromov sense to a singularity model $(M_{\infty},g_{\infty},p_{\infty})_{-\infty < t \leq 0}$ to which the classification in Proposition \ref{classificationsingularity} applies. In particular, $g_{\infty}(t)$ is of the form \eqref{sphericalsymmetrylimitformula}. Arguing as in the proof of Claim \ref{claimderivativezero} and using \eqref{controlinfinitypj} we see that 
\begin{align*}
(1 - (\partial_{y}\phi_{\infty})^{2})(q,0) &= (\phi_{\infty}^{2}k_{12}^{\infty})(q,0) = \lim_{j\rightarrow \infty}\left(4 - \bb_{s}^{2} - 3\frac{\cc^{2}}{\bb^{2}}\right)(\Phi_{j}(q),t_{j}) \\ &= \lim_{j\rightarrow \infty}\left(1 - \bb_{s}^{2}\right)(\Phi_{j}(q),t_{j}) \geq 1 - (\sup_{\R^{4}}\bb^{2}(\cdot,0))\varepsilon^{2} \geq \frac{1}{2},
\end{align*}
\noindent for $\varepsilon$ small enough, where we have used the fact that $\bb$ is uniformly bounded from above. We finally conclude that $\lvert\partial_{y}\phi_{\infty}\rvert < 1/2$ which by the boundary conditions \eqref{smoothnessorigin} and the classification in Proposition \ref{classificationsingularity} implies that $(M_{\infty},g_{\infty}(t))$ is the self-similar shrinking soliton on $\R\times S^{3}$.

\emph{(iv) Classification of singularity models.} According to Proposition \ref{classificationsingularity} if the singularity model is not a family of shrinking cylinders, then it must be a positively curved rotationally symmetric $\kappa$-solution. By the recent classification in \cite{brendle2} we conclude that in this case the singularity model is isometric to the Bryant soliton (up to scaling).
\end{proof}
\subsection{Immortal warped Berger Ricci flows.}
Let $(\R^{4},g(t))_{0\leq t < T}$ be the maximal Ricci flow starting at some $g_{0}\in\Gin$ and suppose that $T < \infty$. 
\begin{proof}[Proof of Theorem \ref{mainimmortalresult}]
Let $\tilde{x}_{2}$, $\tilde{t}$ and $\mu$ be given by Corollary \ref{lowerboundH} and consider a blow-up sequence giving rise to a singularity model $(M_{\infty},g_{\infty}(t),p_{\infty})$. Since by Lemma \ref{claimzozzo} the rescaled geodesic balls are included in $B(\origin,\tilde{x}_{2})$ for $j$ large enough, we can argue exactly as in the proof of Claim \ref{claimderivativezero} and deduce that any singularity model for the flow is in fact flat. This shows that the maximal time of existence cannot be finite.
\end{proof}
\subsection{Type-I Berger Ricci flows contain minimal 3-spheres.}
The existence of sufficiently pinched minimal embedded hyperspheres gives rise to Type-I singularities for (asymptotically flat) rotationally symmetric Ricci flows on $\R^{n}$ \cite{work}. Thus, in general we cannot extend the conclusions of Theorem \ref{maintheoremtype2} and Theorem \ref{mainimmortalresult} to include initial data containing minimal $3$-spheres.
\\While in the SO$(n)$-invariant case no minimal spheres can appear along the flow, in the SU(2)-cohomogeneity 1 setting an analogous property might fail. On the other hand, minimal spheres can disappear in finite time \cite[Proposition 1.7]{work}. 
\\In the following we consider a Type-I warped Berger Ricci flow whose curvature is controlled at spatial infinity uniformly in time. A priori one might expect that there exist examples of Type-I singularities where both $\bb$ and $\cc$ have local minima while the mean curvature of the embedded hyperspheres remains positive. The next result rules out this possibility. We prove that for times close to the maximal time of existence a Type-I warped Berger Ricci flow solution $(\R^{4},g(t))$ must contain minimal $3$-spheres.  
\begin{proof}[Proof of Theorem \ref{maintypeI}] The decay of the curvature and the lower bound for the injectivity radius ensure that \eqref{definitionradiusrho} holds for some sufficiently large radius $\rho$. Therefore Lemma \ref{estimatesforGin} and hence the classification of singularity models in Proposition \ref{classificationsingularity} are satisfied in this setting (see also Remark \ref{classificationextends}). A first consequence of this fact is that the same argument in the proof of Theorem \ref{maintheoremtype2} shows that if the origin is in the singular set (as defined in \cite{type1}), then the singularity cannot be Type-I. 
\\Therefore we only need to consider the case where the curvature stays uniformly bounded in a Euclidean ball $B(\origin,2r)$ for some $r > 0$. Accordingly, there exists $\varepsilon > 0$ such that $\bb(r,t)\geq \cc(r,t)\geq \varepsilon$ for any $t\in [0,T)$. 
\\Assume for a contradiction that there exists a sequence $t_{j}\nearrow T$ such that the mean curvature of hyperspheres $H(\cdot,t_{j})$ is strictly positive on the time slices $\R^{4}\setminus \{\origin\}\times \{t_{j}\}$. From the identity $H = (\log(\bb^{2}\cc))_{s}$ we deduce that 
\begin{equation}\label{inequalityzozza}
\bb^{2}\cc(x,t_{j}) \geq \bb^{2}\cc(r,t_{j})\geq \varepsilon^{3},
\end{equation}
\noindent for any $j$ and for any $x\geq r$. Since Corollary \ref{cadmitslimit} holds in this setting, the singular set contains a point $p\in\R^{4}\setminus B(\origin,2r)$ such that $\bb(p,t) \rightarrow 0$ as $t\nearrow T$ and similarly for $\cc$ by Lemma \ref{rotationalsymmetry}. For if such $p$ did not exist, then by the first estimate in Lemma \ref{estimatesforGin} applied to the region $B(\origin,\rho)\setminus B(\origin,r)\times [0,T)$, the curvature would be bounded as $t\nearrow T$ which is not possible by \cite{shi}. This contradicts the inequalities in \eqref{inequalityzozza}.
\end{proof}
\begin{remark}
The proof actually shows that $H$ has to change sign for times close to the maximal time of existence. Equivalently, the Ricci flow solution contains neck-like regions that pinch off in finite time at a Type-I rate.
\end{remark}
\section{Some applications}
In this section we provide two simple applications of the main results. On the one hand we rule out the existence of Taub-NUT like shrinking solitons on $\R^{4}$. On the other hand, we completely classify Ricci flows of nonnegatively curved warped Berger metrics. 
\subsection{Non existence of Taub-NUT like shrinking solitons.}
Theorem \ref{maintheoremtype2} immediately implies that there are no warped Berger shrinking solitons with no necks and bounded by a round cylinder at infinity. Namely, we have the following 
\begin{corollaryy}\label{corollaryshrinking1}
The set $\G$ does not contain shrinking Ricci solitons.
\end{corollaryy}
Recently, Appleton found non-collapsed Taub-NUT like gradient \emph{steady} solitons on $\R^{4}$ \cite{appleton}. It is straightforward to check that these solitons belong to $\Gin$. Indeed the curvature decays linearly at spatial infinity and both the warping functions $\bb$ and $\cc$ are increasing in space. According to Theorem \ref{mainimmortalresult} there are no shrinking solitons on $\R^{4}$ analogous to the steady ones constructed in \cite{appleton}. More precisely, we have shown the following:
\begin{corollaryy}\label{corollaryshrinking2}
The set $\Gin$ does not contain shrinking Ricci solitons. 
\end{corollaryy}
We note that by \cite{kotschwar} it is known that there do not exist complete non-trivial \emph{rotationally invariant} shrinking soliton structures on $\R^{4}$.

\subsection{Ricci flow of Berger metrics with nonnegative curvature.}
By combining Theorem \ref{maintheoremtype2} and Theorem \ref{mainimmortalresult} we are able to classify Ricci flows evolving from complete warped Berger metrics with bounded nonnegative curvature operator. 
We recall that by \cite{positivecurvature} the curvature operator stays nonnegative along the Ricci flow in any dimension. 
\begin{proof}[Proof of Corollary \ref{mainpositiveresult}] If $g_{0}$ is a complete warped Berger metric with bounded nonnegative curvature, then the injectivity radius of $g_{0}$ is positive and $\bb_{ss} \leq 0$ and $\cc_{ss}\leq 0$. By completeness the latter condition implies that both $\bb_{s}$ and $\cc_{s}$ are nonnegative. Thus there exists a positive (possibly infinite) quantity $\mu \doteq \lim_{x\rightarrow \infty}\bb(x,0)$. 

If $\mu < \infty$, i.e. $\bb(\cdot,0)$ is bounded on $\R^{4}$, then $g_{0}$ belongs to $\G$ and the conclusions of Theorem \ref{maintheoremtype2} and Theorem \ref{maintheoremtype2bis} apply. In particular, there exists a sequence $t_{j}\nearrow T$ such that the rescaled Ricci flows $(\R^{4},g_{j}(t),\origin)$ defined by $g_{j}(t)\doteq R_{g(t_{j})}(\origin)g(t_{j} + (R_{g(t_{j})}(\origin))^{-1}t)$ converge to the Bryant soliton in the Cheeger-Gromov sense. Given any other sequence $\tilde{t}_{j}\nearrow T$, from the trace of the Harnack estimate \cite{harnack} we derive that
\[
R(\origin,\tilde{t}_{j}) \geq \frac{t_{j}}{\tilde{t}_{j}}R(\origin,t_{j}),
\]
\noindent up to reordering the indices. Therefore we conclude that dilations of the Ricci flow by factors $R(\origin,\tilde{t}_{j})$ still give rise to the Bryant soliton.

If $\mu = \infty$, then consider the Ricci flow solution starting at $g_{0}$ and pick $0 < t_{0} < T$. The vertical sectional curvatures decay to zero at infinity being the spatial derivatives $\bb_{s}(\cdot,t_{0})$ and $\cc_{s}(\cdot,t_{0})$ decreasing and nonnegative. In particular $\bb_{ss}/\bb(\cdot,t_{0})$ (and $\cc_{ss}/\cc(\cdot,t_{0})$ as well) is integrable. By the same argument we used to prove Claim \ref{secondderivativesdecay} we get $\bb_{ss}/\bb(x,t_{0})\rightarrow 0$ at infinity and similarly for $\cc_{ss}/\cc(x,t_{0})$. Therefore $g(t_{0})\in\Gin$ and we can apply Theorem \ref{mainimmortalresult}.

\end{proof}
\bibliographystyle{alpha}
\bibliography{referenceswarped2}
\end{document}